\documentclass[aps,prd,showpacs,amssymb,superscriptaddress,nofootinbib,onecolumn,notitlepage, 12pt]{revtex4-1}
\usepackage{amsmath,amsfonts,amsthm,amssymb,verbatim,color}
\usepackage{bbm}
\usepackage{siunitx}
\usepackage{graphicx}
\usepackage{braket}
\usepackage{float}
\usepackage{epstopdf}
\usepackage{hyperref}
\usepackage{multirow}
\usepackage{tabularx}
\usepackage{subcaption}
\usepackage{scalerel}
\usepackage{bm}
\usepackage[top=0.6in, bottom=0.8in, left=0.7in, right=0.7in]{geometry}
\usepackage[titletoc,toc,page]{appendix}
\usepackage{times}
\newtheorem{theorem}{Theorem}

\newtheorem{remark}{Remark}
\newtheorem{definition}{Definition}

\newcolumntype{Y}{>{\centering\arraybackslash}X}

\DeclareMathOperator*{\argmin}{arg\,min}

\theoremstyle{definition}

\theoremstyle{Theorem}

\theoremstyle{Theorem}

\theoremstyle{Theorem}

\theoremstyle{definition}

\begin{document}
	
	\title{Optimally weighted loss functions for solving PDEs with Neural Networks}
	\author{Remco van der Meer}
	\affiliation{CWI, Science Park 123, 1098 XG Amsterdam, Netherlands}	
	\affiliation{Delft University of Technology, Van Mourik Broekmanweg 6, 2628 XE Delft, Netherlands}	
	\author{Cornelis Oosterlee}
	\affiliation{CWI, Science Park 123, 1098 XG Amsterdam, Netherlands}
	\affiliation{Delft University of Technology, Van Mourik Broekmanweg 6, 2628 XE Delft, Netherlands}
	\author{Anastasia Borovykh}
	\affiliation{CWI, Science Park 123, 1098 XG Amsterdam, Netherlands}
	\affiliation{Department of Computing, Imperial College London, London SW7 2AZ, UK}
	
	\begin{abstract}
	Recent works have shown that deep neural networks can be employed to solve partial differential equations, giving rise to the framework of physics informed neural \mbox{networks \cite{PhysInformedI}}. We introduce a generalization for these methods that manifests as a scaling parameter which balances the relative importance of the different constraints imposed by partial differential equations. A mathematical motivation of these generalized methods is provided, which shows that for linear and well-posed partial differential equations, the functional form is convex. We then derive a choice for the scaling parameter that is optimal with respect to a measure of relative error. Because this optimal choice relies on having full knowledge of analytical solutions, we also propose a heuristic method to approximate this optimal choice. The proposed methods are compared numerically to the original methods on a variety of model partial differential equations, with the number of data points being updated adaptively. For several problems, including high-dimensional PDEs the proposed methods are shown to significantly enhance accuracy.\\
	\end{abstract}
	\maketitle
	
\section{Introduction}

Advances in computing power and rapid growth of available data in recent years have invigorated the field of machine learning and data science. Although the theory to train deep learning models has been available since the early 60's, only recently has it become possible to train them on commonly available hardware \cite{2014arXiv1404.7828S}. This has led to exceptional achievements in a wide range of problems, including image recognition, natural language processing, genomics and reinforcement learning. Aside from these empirical achievements, the theoretical understanding of such methods is also advancing rapidly. This has sparked interest in solving more fundamental problems by utilizing these methods. 

Previous work has shown the successes of Bayesian approaches for learning partial differential equation (PDE) representations in linear settings \cite{owhadi2015bayesian}, \cite{raissi2017machine}; extensions to nonlinear regimes introduced certain limitations \cite{owhadi2015brittleness}, \cite{raissi2018numerical}. Motivated by the universal approximation theorems \cite{hornik1989multilayer, hornik1990universal}, recent studies have considered a different approach for solving (non)linear PDEs utilizing deep neural networks \cite{1997physics...5023L, 2018JCoPh.375.1339S, PhysInformedI, 2019arXiv190407200D, DNN_FEM, han2018solving, chan2019machine}. Specifically, the works of \cite{2018JCoPh.375.1339S, PhysInformedI} minimize a loss function which represents the PDE constraints, while the work of \cite{han2018solving} uses the backward stochastic differential equation (BSDE) representation. The work of \cite{chan2019machine} studies the performance of different network architectures. These studies present several different advantages that neural network based methods have over classical numerical methods. The most notable advantage mentioned is that neural network based methods do not require any form of discretization. Constructing a mesh can be especially prohibitive for high-dimensional problems, problems with complex geometries, or domains with tiny structures. By alleviating this requirement, neural network based methods may have great potential to outperform existing methods in these areas. As shown in e.g. \cite{grohs2018proof} \cite{berner2020analysis} \cite{jentzen2018proof}, \cite{darbon2020overcoming} neural networks can provably overcome the curse of dimensionality in certain classes of PDEs. The study of \cite{1997physics...5023L} mentions additional advantages, which include the differentiability of the obtained solution and the possibility to efficiently parallelize the methods. While neural networks have been succesful in solving a variety of PDEs, certain challenges remain. As mentioned in e.g. \cite{wang2020understanding}  neural networks can suffer from unbalanced back-propagated gradients. Alternatively, when the PDE is only partially known and data-collection is expensive, it is not trivial how to employ neural networks; the work of \cite{chakraborty2021transfer} proposes transfer learning from low- to high-fidelity models as a solution.

The key concept brought forward in these studies is to approximate (part of) the solution of the PDE one aims to solve with a deep neural network. Consider a general PDE for the scalar function $u(\bm{x}): \mathbb{R}^d \rightarrow \mathbb{R}$ on the domain $\Omega\subset\mathbb{R}^d$ given by
\begin{align}
\begin{cases}
\begin{aligned}
N(\bm{x}, u) &= F(\bm{x}) &&\text{in } \Omega,\\
B(\bm{x}, u) &= G(\bm{x}) &&\text{on } \partial\Omega,
\end{aligned}
\end{cases}\label{eq:pde_def_1}
\end{align}
with $\bm{x}\in\Omega\subset\mathbb{R}^d$ and $N$ and $B$ the differential operators on the interior and boundary, resp., and $F$, $G$ are source functions. These operators and functions define constraints on $u$ that must be satisfied to solve the PDE. The neural network-based methods employ a deep neural network of which the input layer has $d$ neurons and the output layer has a single neuron representing the entire solution of the PDE. To solve the PDE, the neural network must discover solutions that satisfy the constraints imposed by the PDE and the boundary conditions. This discovery is done without using any explicit information about the true solution; only the PDE and the boundary conditions that must be satisfied are provided. Therefore, these methods can be categorized as unsupervised learning methods. 

There are different ways to satisfy the constraints imposed by the PDE and the boundary conditions. The study of \cite{1997physics...5023L} treats the boundary conditions of the PDE as a hard constraint by constructing an auxiliary function that satisfies the boundary conditions. The remaining constraint that is imposed by the PDE itself is treated as a soft constraint, which is approximately satisfied by minimizing some loss function. The studies \cite{2018JCoPh.375.1339S, PhysInformedI, 2019arXiv190407200D} use a different approach and treat all the constraints as soft constraints. This is implemented by constructing a single loss function in which all of these constraints are included. These approaches are more general, as they do not require special functions that depend on the problem one aims to solve. The result is that the neural network directly serves as the approximation of the entire solution; feeding a location\,$\bm{x}$ into the neural network results in an output $\hat{u}(\bm{x})$ that approximates the true \mbox{solution $u(\bm{x})$}. 

To solve the general PDE given in eq. \ref{eq:pde_def_1}, the authors of \cite{PhysInformedI} suggest rewriting the PDE into the form
\begin{align}
\begin{cases}
\begin{aligned}
\mathcal{N}(\bm{x}, u) &:= N(\bm{x}, u) - F(\bm{x}) = 0 &&\text{in } \Omega,\\
\mathcal{B}(\bm{x}, u) &:= B(\bm{x}, u) - G(\bm{x}) = 0 &&\text{on } \partial\Omega,
\end{aligned}
\end{cases}\label{eq:pde_def_1_rewritten}
\end{align}
and minimizing the loss functional
\begin{equation}
	L(\hat{u}) = \frac{1}{n_I}\sum_{i=1}^{n_I} \mathcal{N}(\bm{x}_i^I, \hat{u})^2 + \frac{1}{n_B}\sum_{i=1}^{n_B} \mathcal{B}(\bm{x}_i^B, \hat{u})^2, \label{eq:loss_def_1}
\end{equation}
with a deep neural network. Here, $x_i^I \in \Omega$ and $x_i^B \in \partial\Omega$ are collocation points. The authors of \cite{2018JCoPh.375.1339S, PhysInformedI, 2019arXiv190407200D} provide empirical evidence that for several instances of well-posed problems, e.g. the Burgers' equation and the Poisson equation, these methods yield accurate results. Although these empirical results are promising, it is still poorly understood why or when these methods work. The authors of \cite{PhysInformedI} suggest that well-posedness may be what determines whether or not these methods work. However, in Section \ref{sec:results} we show that instances of well-posed PDEs for which minimizing \mbox{eq. \ref{eq:loss_def_1}} leads to inaccurate approximations are not uncommon.

The main aim of this work is to augment the generality of the methods introduced in \cite{PhysInformedI} by extending the class of problems that may be solved with these methods. To this end, we aim to address the aforementioned problems. Herein our focus will be on research around the neural network engine; The scientific community has spent decades developing sophisticated training algorithms, and neural network training techniques are expected to improve further. Our focus is on translating the PDEs to a minimization problem that is well suited for neural networks to solve. 

In Section \ref{sec:motivation}, we give a mathematical motivation for the original methods that were introduced in \cite{PhysInformedI}. Here, we show that in the asymptotic regime of large neural networks, these methods should indeed be able to solve linear, well-posed problems. Then, Section \ref{sec:modification} introduces a scaling parameter which generalizes these methods. This parameter is optimized with respect to a measure of relative error. Section \ref{sec:normalization} proposes a heuristic method to perform this optimization which does not require the explicit knowledge of the analytical solution. These theoretical results are analyzed numerically in Section \ref{sec:results}. Here, several linear well-posed model PDEs such as the Laplace equation and the convection-diffusion equation are considered. Instances of these PDEs for which the original methods worked well and instances for which they do not are both used to examine the impact of the proposed generalization. 

\section{Method Motivation} \label{sec:motivation}

This section aims to provide a mathematical motivation for the use of the methods introduced in \cite{PhysInformedI}. This is done by showing that the loss function of eq. \ref{eq:loss_def_1} can be viewed as a Monte-Carlo approximation of a functional, and subsequently by showing that this loss functional satisfies several highly desirable properties for well-posed linear PDEs. Although these properties are challenging to exploit directly, they can be used to justify the use of approximate methods to solve such problems. 
Loss functionals are challenging to compute exactly and therefore rarely used in practice for the purpose of training neural networks. However, loss functionals are significantly easier to analyze mathematically; by focusing on them instead of their Monte-Carlo counterparts, one can avoid all problems related to the distribution of the collocation points. To emphasize the difference, loss functionals are denoted with a hat throughout this work, while their Monte-Carlo approximations are denoted without one. 

For the sake of generality, the analysis is performed on a generalized loss functional. Instead of considering the mean squared error, i.e. the second power of the $L^2$ norm, the $p$-th power of the $L^p$ norm is considered for $p\geq 1$. The resulting loss functional is given by
\begin{equation}
\hat{L}(\hat{u}) = \frac{1}{|\Omega|}\int_\Omega \left|\mathcal{N}(\bm{x}, \hat{u})\right|^p d\bm{x} + \frac{1}{|\partial\Omega|}\int_{\partial\Omega} \left|\mathcal{B}(\bm{x}, \hat{u})\right|^p d\bm{x}_\Gamma. \label{eq:loss_def_3}
\end{equation}
Note that for $p=2$ the Monte-Carlo approximation of this functional is given by eq. \ref{eq:loss_def_1}. For notational convenience, the constants present in eq. \ref{eq:loss_def_3} are replaced by the constants $0 < c_1,c_2 \in \mathbb{R}$. This generalization leads to the loss functional
\begin{equation}
\hat{L}(\hat{u}) = c_1\int_\Omega \left|\mathcal{N}(\bm{x}, \hat{u})\right|^p d\bm{x} + c_2\int_{\partial\Omega} \left|\mathcal{B}(\bm{x}, \hat{u})\right|^p d\bm{x}_\Gamma. \label{eq:loss_def_4}
\end{equation}
The integrals present in eqs. \ref{eq:loss_def_3} and \ref{eq:loss_def_4} are labeled as
\begin{align}
\hat{L}_I(\hat{u}) &:= \int_\Omega \left|\mathcal{N}(\bm{x}, \hat{u})\right|^p d\bm{x} \label{eq:interiorLoss}\\
\hat{L}_B(\hat{u}) &:= \int_{\partial\Omega} \left|\mathcal{B}(\bm{x}, \hat{u})\right|^p  d\bm{x}_\Gamma, \label{eq:boundaryLoss}
\end{align}
and are referred to as the interior and boundary loss, respectively.

The functional given in eq. \ref{eq:loss_def_3} has a much clearer correspondence to the PDE given in eq. \ref{eq:pde_def_1} than the Monte-Carlo approximation. Since the loss functional is positive, and only zero when $\mathcal{N}$ and $\mathcal{B}$ are zero everywhere, its global minimizer coincides with the solution of the PDE. This is a property that the Monte-Carlo approximation of the loss functional lacks, since minimizing that loss function only ensures that the PDE is satisfied on a finite number of collocation points. 

Despite this, minimizing the loss function of eq. \ref{eq:loss_def_1} to solve a PDE can still be motivated with the properties of the loss functional of eq. \ref{eq:loss_def_3} for well-posed PDEs. Definition \ref{def:well-posed} gives the formal definition of well-posedness as given in \cite{Wesseling}.
\begin{definition}
	\label{def:well-posed}
	Consider a PDE of the form 
	\begin{align}
	\begin{cases}
	\begin{aligned}
	\mathcal{N}(\bm{x}, u) := N(\bm{x}, u) - F(\bm{x}) &= 0 &&\text{in } \Omega,\\
	\mathcal{B}(\bm{x}, u) := B(\bm{x}, u) - G(\bm{x}) &= 0 &&\text{on } \partial\Omega,
	\end{aligned}
	\end{cases}
	\end{align}
	on the finite and sufficiently smooth domain $\Omega$, with $N$, $B$ the operators that define the PDE, and $F$, $G$ source functions. Such a PDE is called well-posed if for all $F, G$ there exists a unique solution, and if for every two sets of data $F_1, G_1$ and $F_2, G_2$, the corresponding solutions $u_1$ and $u_2$ satisfy
	\begin{equation}
	\|u_1 - u_2\| \leq C \{\|F_1 - F_2\| + \|G_1 - G_2\|\}
	\end{equation}
	for some fixed, finite constant $C\in\mathbb{R}$. Such a constant $C$ will be referred to as the Lipschitz constant of the PDE. Here, $\|\cdot\|$ denotes the $L^1$ norm. 
\end{definition}

Motivating the use of the Monte-Carlo approximation of the loss functional of eq. \ref{eq:loss_def_3} requires a statement regarding the approximate optimization of these functionals; so far, it is only clear that minimizing the loss functional exactly yields a solution of the PDE. However, when one minimizes the Monte-Carlo approximation instead, it is highly unlikely that this exact minimum is attained. Even if one were able to find a solution for which the approximated loss function zeroes out completely, the loss functional would likely remain nonzero. The following theorem bridges the gap between this approximate and exact optimization of the loss functionals. 

\begin{theorem}
	\label{thm:defaultconsistent}
	Consider the well-posed PDE of order $k$ given by
	\begin{equation}
	\begin{cases}
	\begin{aligned}
	\mathcal{N}(\bm{x}, u) := N(\bm{x}, u) - F(\bm{x}) &= 0 &&\text{in } \Omega,\\
	\mathcal{B}(\bm{x}, u) := B(\bm{x}, u) - G(\bm{x}) &= 0 &&\text{on } \partial\Omega.
	\end{aligned}
	\end{cases} \label{eq:pde_def_2}
	\end{equation}
	Let the exact solution of this PDE be given by $u$ and let the loss functional be given by eq. \ref{eq:loss_def_4} for some fixed $p\geq 1$ and $c_1, c_2 > 0$. Consider some approximate solution $\hat{u}$ of which the first $k$ (partial) derivatives exist and have finite $L^p$ norm. Then, for any $\epsilon > 0$ there exists a $\delta > 0$ such that for the approximate solution $\hat{u}$,
	\begin{equation}
	\hat{L}(\hat{u}) < \delta \implies \|\hat{u} - u\| < \epsilon.\label{eq:thmdefaultconsistent}
	\end{equation}
\end{theorem}


\begin{proof}
	Let $\epsilon > 0$ be arbitrary. The aim is to find a $\delta$ for which eq. \ref{eq:thmdefaultconsistent} holds. Because the PDE given in eq. \ref{eq:pde_def_2} is assumed to be well-posed, there exists a finite Lipschitz constant $C$ for this particular PDE. Given this Lipschitz constant, and given the constants $p, c_1$ and $c_2$, choose  
	\begin{equation}
	\delta := \epsilon^p\left[C\left(c_1 ^ {-\frac{1}{p}}|\Omega|^{1 - \frac{1}{p}}  + c_2 ^ {-\frac{1}{p}}|\partial\Omega|^{1 - \frac{1}{p}} \right)\right]^{-p}.
	\end{equation}
	and let $\hat{u}$ be some approximate solution of the PDE for which $\hat{L}(\hat{u}) < \delta$. Because $\hat{u}$ may not exactly satisfy the operators $N$ and $B$ of eq. \ref{eq:pde_def_2}, one can write
	\begin{align}
	\begin{split}
	N(\bm{x}, \hat{u}) = F + \hat{F},\;\; B(\bm{x}, \hat{u}) = G + \hat{G}.
	\end{split}
	\end{align}
	In other words, $\hat{u}$ satisfies a perturbed version of the PDE. Because the PDE is well-posed, it follows that
	\begin{align}
	\begin{split}
	\|\hat{u} - u\| &< C\{\|(\hat{F} + F) - F\| + \|(\hat{G} + G) - G\|\}= C\{\|\hat{F}\| + \|\hat{G}\|\}.
	\end{split}
	\end{align}		
	Using H\"older's inequality \cite{Holder}, one finds for $\|\hat{F}\|$ the upper bound
	\begin{align}	
	\begin{split}
	\|\hat{F}\| &= \int_\Omega |\hat{F}(\bm{x})|d\bm{x}\leq |\Omega|^{1 - \frac{1}{p}} \left[\int_\Omega |\hat{F}(\bm{x})|^pd\bm{x}\right]^{\frac{1}{p}} = |\Omega|^{1 - \frac{1}{p}} \hat{L}_I(\hat{u})^\frac{1}{p}.
	\end{split}\label{eq:F_inequality_1}
	\end{align}	
	Similarly, $\|\hat{G}\|$ can be bounded from above by
	\begin{align}	
	\begin{split}
	\|\hat{G}\| &= \int_{\partial\Omega} |\hat{G}(\bm{x})|d\bm{x}_\Gamma\leq |\partial\Omega|^{1 - \frac{1}{p}} \left[\int_{\partial\Omega} |\hat{G}(\bm{x})|^pd\bm{x}_\Gamma\right]^{\frac{1}{p}}= |\partial\Omega|^{1 - \frac{1}{p}} \hat{L}_B(\hat{u})^\frac{1}{p}.
	\end{split}
	\end{align}			
	Combining these results gives
	\begin{align}
	\begin{split}
	\|\hat{u} - u\| &< C\{\|\hat{F}\| + \|\hat{G}\|\}\leq C\{|\Omega|^{1 - \frac{1}{p}} \hat{L}_I(\hat{u})^\frac{1}{p} + |\partial\Omega|^{1 - \frac{1}{p}} \hat{L}_B(\hat{u})^\frac{1}{p}\}\\
	& \leq C\left[c_1 ^ {-\frac{1}{p}}|\Omega|^{1 - \frac{1}{p}}  + c_2 ^ {-\frac{1}{p}}|\partial\Omega|^{1 - \frac{1}{p}} \right]\hat{L}(\hat{u})^\frac{1}{p}.
	\end{split}
	\end{align}
	Finally, applying the inequality $\hat L(\hat{u}) <\delta$ yields
	\begin{align}
	\|\hat{u} - u\| & < C\left[c_1 ^ {-\frac{1}{p}}|\Omega|^{1 - \frac{1}{p}}  + c_2 ^ {-\frac{1}{p}}|\partial\Omega|^{1 - \frac{1}{p}} \right]\delta^\frac{1}{p}\nonumber= \epsilon,
	\end{align}
	completing the proof.
\end{proof}

\begin{remark}
When one approximates the solution $u$ with a neural network, the activation function determines whether the assumptions made in the Theorem~\ref{thm:defaultconsistent} are satisfied. Activation functions that are in $C^\infty$ such as the hyperbolic tangent result in neural networks that satisfy all assumptions.
\end{remark}

Because Monte-Carlo approximations have a high probability of being close in value to the functions they approximate, it is unlikely that a small approximated value arises when the functional itself is large, given a sufficient number of collocation points. Theorem \ref{thm:defaultconsistent} can then be applied to conclude that a low loss function implies with high probability that the approximations are accurate, validating the use of the approximation for well-posed problems. This theorem also plays an important role in justifying the use of neural networks to minimize these loss functions or functionals; although neural networks are universal approximators, they are unable to exactly represent many functions, and may therefore not be able to exactly reach the minimum of the loss functions or functionals.  Theorem \ref{thm:defaultconsistent} shows that this is not a problem, as approximate optimization is sufficient. Thus, Theorem \ref{thm:defaultconsistent} can be used to show that when neural networks reach a low loss value, the approximation they define is likely accurate. 

This theorem cannot be used to show that it is reasonable to expect neural networks to reach such small loss values. Due to the significant gap between neural network theory and practice, it is beyond the scope of this work to formally prove this. Therefore, instead of looking at sufficient conditions to guarantee successful network training, one might look at necessary conditions. To this end, consider a loss functional for which local minima exist. It can be shown that the corresponding loss function that would be used to train the neural network would also have local minima for large neural networks with sufficiently smooth activation functions.

By the universal approximation theorem, any smooth function and its derivatives can be approximated to arbitrary accuracy on compact domains by large enough neural networks with sufficiently smooth activation functions. This also means that local minimum points can be approximated to arbitrary accuracy by such neural networks. Furthermore, note that these neural networks are continuous in their parameters, meaning that a small change in the parameters will result in a small change of the function that the network defines, as well as its derivatives. By assumption, functions close to the local minima have a loss that is larger than the local minimum. Thus, the parameter vectors that map to this region of functions also have larger losses. In other words, the local minimum also exists in the network's parameters. 

Popular training algorithms such as gradient descent do not handle local minima well. Though there are methods to escape local minima, these methods typically have a performance cost associated and are not guaranteed to find the global minimum. It is therefore very important that the loss functional does not have local minima. Fortunately, showing that the loss functional does not have local minima can be done with relative ease for linear PDEs. In fact, it can be shown that a much stronger property holds for these PDEs. This is done in Theorem \ref{thm:convexity}. 

\begin{theorem} \label{thm:convexity}
    Consider a linear PDE of the form given in eq. \ref{eq:pde_def_2} of order $k$. Then, the loss functional $L: \mathcal{F}\rightarrow\mathbb{R}$ defined in eq. \ref{eq:loss_def_4}, where $\mathcal{F}$ is the space of functions whose partial derivatives up to order $k$ exist and have finite $L^p$ norm, is convex.
\end{theorem} 

\begin{proof}
	To show that the loss functional is convex, it must be shown that for any two functions $u_1, u_2\in \mathcal{F}$ and for any $t\in[0, 1]$, the inequality
	\begin{equation}
		\hat{L}(tu_1 + (1-t)u_2) \leq t\hat{L}(u_1) + (1 - t)\hat{L}(u_2)
	\end{equation} 
	holds. To this end, let $u_1, u_2\in \mathcal{F}$ and $t\in[0, 1]$ be arbitrary. For notational convenience, the interpolation is labeled $v(t) := tu_1 + (1-t)u_2$. Note that $v(t) \in \mathcal{F}$.  
	The loss functional at the interpolation is given by $\hat{L}(v(t)) = c_1 \hat{L}_I(v(t)) + c_2 \hat{L}_B(v(t))$, with $c_1, c_2 > 0$. The interior loss is considered first. It is given by
	\begin{align}
	\begin{split}
		\hat{L}_I(v(t)) &= \int_\Omega \left|\mathcal{N}(v(t))\right|^pd\bm{x}= \int_\Omega \left|N(v(t)) - F(\bm{x})\right|^pd\bm{x}. 
		\end{split}
	\end{align}
	For linear PDEs, $N$ is linear, and hence, using additionally the triangle inequality, this can be bounded by
	\begin{align}
	\begin{split}
	\hat{L}_I(v(t)) &= \int_\Omega \left|N(tu_1 + (1-t)u_2) - F(\bm{x})\right|^pd\bm{x}
	\leq \int_\Omega \left(t|\mathcal{N}(u_1)| + (1-t)|\mathcal{N}(u_2)|\right)^pd\bm{x}.
	\end{split}
	\end{align}	
Because $|x|^p$ is a convex function for $p\geq 1$, it holds that 
\begin{equation}
    \left(t|\mathcal{N}(u_1)| + (1-t)|\mathcal{N}(u_2)|\right)^p \leq t|\mathcal{N}(u_1)|^p + (1-t)|\mathcal{N}(u_2)|^p.
\end{equation}
Thus, it follows that
\begin{align}
\begin{split}
	\hat{L}_I(v(t)) &\leq \int_\Omega \left(t|\mathcal{N}(u_1)|^p + (1-t)|\mathcal{N}(u_2)|^p\right)d\bm{x} \leq t\int_\Omega |\mathcal{N}(u_1)|^pd\bm{x} + (1-t)\int_\Omega |\mathcal{N}(u_2)|^p d\bm{x}\\
	& = t\hat{L}_I(u_1) + (1 - t)\hat{L}_I(u_2).
\end{split}\label{eq:convexproof1}
\end{align}	
The same reasoning can be applied to show that when the boundary operator $B$ is linear,
\begin{align}
\hat{L}_B(v(t)) &\leq t\hat{L}_B(u_1) + (1 - t)\hat{L}_B(u_2). \label{eq:convexproof2}
\end{align}	
Eqs. \ref{eq:convexproof1} and \ref{eq:convexproof2} can then be combined to complete the proof. 	
\end{proof}

Thus, for any linear PDE, the loss functional defined in eq. \ref{eq:loss_def_4} is convex. A direct consequence is that the loss functional has no local minima. Therefore, any local minima that one encounters while training must necessarily be a product of the configuration of the employed neural network. Although it is well known that neural networks convert convex loss functionals into non-convex problems in parameter space, many recent studies indicate that this non-convexity can be overcome when dealing with classification problems. For instance, the study of \cite{LossSurfNoBarrier} empirically shows that local minima can be connected with paths through parameter space for which the loss stays low. Other work has also shown that in the limit of the number of nodes tending to infinity, convexity can be assumed \cite{bach2017breaking}.
Even though classification problems cannot be directly compared to solving PDEs, one would expect these problems to share some characteristics; after all, in function space, classification problems are convex, just like solving linear PDEs.

It is important to note that Theorem \ref{thm:convexity} required the assumption of linearity; for nonlinear PDEs, it is not clear whether the loss functional is convex. For some nonlinear PDEs the loss functional may even have local minima. In these cases, one might consider using global optimization algorithms instead of local algorithms. 

\section{Loss Functional Modification} \label{sec:modification}

The previous section has provided a motivation for the use of the methods introduced in the study of \cite{PhysInformedI}. This was done by reframing the process of solving a PDE as a minimization problem. 

In this section, the methods are considered from a different perspective to highlight an important choice to be made. In this section a proposed modification to the loss functionals is discussed. This modification results in a more general formulation that contains an additional hyperparameter. This parameter may be chosen in a way that is tailored to the specific PDE one aims to solve, resulting in a method that is more generally applicable. 

\subsection{Multi-Objective Optimization}

Recall that solving a PDE is equivalent to finding a function satisfying certain constraints. So far, these constraints have been included in the loss functionals, such as in eq. \ref{eq:loss_def_4}. Because there is more than a single constraint in this functional, minimizing it can be viewed as a scalarization of the multi-objective optimization problem given by
\begin{equation}
    \argmin_{\hat{u}} \left(\hat{L}_I(\hat{u}), \hat{L}_B(\hat{u})\right).
\end{equation}
In other words, multiple objectives were merged into a single objective function. This method only works under the assumption that a very low loss can be reached. When low losses are reached, the exact value of the individual losses is not important, as both are necessarily small. However, given the limitations of the capacity of finite neural networks, it may not always be possible to reach such low values. If the different constraints that must be satisfied are not of similar difficulty, then the neural network might not optimize all of them equally. In some cases the network might even sacrifice one objective to better optimize the others. To avoid such problems, the multi-objective optimization problem should be treated more carefully. 

Viewed from this perspective, it makes sense to apply a weighted scalarization. The question how these weights should be chosen will be addressed later. To perform this weighted scalarization, we introduce the scaling parameter $\lambda\in(0, 1)$, such that the modified loss functional is given by
\begin{align}
	\hat{L}(\hat{u}) &:= \lambda \hat{L}_I(\hat{u}) + (1 - \lambda) \hat{L}_B(\hat{u}) \label{eq:loss_def_5}\\
	&= \lambda \int_\Omega \left|\mathcal{N}(\bm{x}, \hat{u})\right|^p d\bm{x} + (1 -  \lambda)\int_{\partial\Omega} \left|\mathcal{B}(\bm{x}, \hat{u})\right|^p d\bm{x}_\Gamma.\nonumber
\end{align}
This hyperparameter $\lambda$ will be referred to as the loss weight. Note that theorem \ref{thm:defaultconsistent} holds for the redefined loss functional of eq. \ref{eq:loss_def_5} for any constant $\lambda$. With this modification, a scaled version of the original loss function can be recovered by setting
\begin{equation}
	\lambda = \frac{|\partial\Omega|}{|\partial\Omega| + |\Omega|}	,\label{eq:lambda_orig}
\end{equation}
and applying a Monte-Carlo approximation, though, there is little motivation for this particular choice of $\lambda$. This family of loss functions can be used to solve a wider range of problems. This is a valuable property, since the original methods were designed to be general. To exploit this, the choice of $\lambda$ is investigated in sections \ref{subsec:constlam} and \ref{sec:normalization}. In the first of these sections, $\lambda$ is treated as a constant. In Section \ref{sec:normalization}, it is treated as a function instead, allowing the hyperparameter to have a deeper effect when training the neural networks. 

\subsection{Optimized Loss Weights} \label{subsec:constlam}
This section investigates the previously introduced hyperparameter $\lambda$ in more detail, under the assumption that it is kept constant. This includes the derivation of an optimal choice for this loss weight with respect to a specific error measure. 

Optimizing $\lambda$ is not a straightforward process, mainly because it is unclear with respect to which quantity this parameter should be optimized. Since the described methods all work by minimizing a loss function to solve a PDE, one may consider choosing $\lambda$ to optimize the relation between the loss functional and the solution of the PDE; a weak relation between the accuracy and the loss functional might result in low losses, but also in inaccurate approximations. On the other hand, a strong relation between these variables would imply that the neural network spends its capacity as effectively as possible. 

The strength of this relation can be probed by comparing the loss of a function to the error of this function. The aim is to find the loss functional of which the minimizer is as close to the solution $u$ as possible. Without explicit knowledge about the network configuration, it is a-priori not clear which functions can and cannot be reached by the neural network, and therefore it is also not clear which loss values may be achieved. Here, a helpful assumption can be made. Under the current setting, the total loss functional interpolates two different functionals with weight $\lambda$. If the value of $\lambda$ is changed such that the interior loss becomes more important, then the neural network might be able to reduce the interior loss at the expense of some accuracy at the boundary. As a result, the total loss value may not change significantly if $\lambda$ is changed. This provides an important tool to perform the optimization: the optimal loss value can be held constant. 

One could consider minimizing the absolute error of the solution for a given, fixed loss value. Although this would yield the tightest upper bound on the error for any given loss value, practical results may not benefit much from this: because neural networks cannot approximate every function equally easily, certain error distributions may not arise in practice. In particular, the results of section \ref{sec:results} suggest that the error of the derivatives tends to be highly correlated with the derivatives of the true solution. Depending on the shape of the true solution of a PDE, such error distributions often lead to absolute errors that are significantly smaller than the worst-case scenario that would be assumed to derive absolute error bounds.

For this reason, instead of considering the absolute error, one could consider minimizing some form of relative error, both of the solution itself and of its derivatives. More formally, this results in the following definition. 
\begin{definition}
	A candidate solution $\hat{u}$ is called $\epsilon$-close to the true solution $u$ if it satisfies
	\begin{equation}
		\left|\partial_{\bm{x}}^{\bm{\gamma}} u(\bm{x}) - \partial_{\bm{x}}^{\bm{\gamma}}\hat{u}(\bm{x})\right| \leq \epsilon\left|\partial_{\bm{x}}^{\bm{\gamma}}u(\bm{x})\right| \label{eq:epsClose}	\end{equation}
	for all $\bm{x}\in\mathbb{R}^d$ and for any vector $\gamma\in\mathbb{R}^d$ with positive elements $\gamma_i \geq 0$. Here, $\partial_{\bm{x}}^{\bm{\gamma}}$ denotes the parametrized partial derivative given by
	\begin{equation}
		\partial_{\bm{x}}^{\bm{\gamma}}u := \left(\prod_{i=1}^{d} \frac{\partial^{\gamma_i}}{\partial x_i^{\gamma_i}}\right)u.
	\end{equation}
\end{definition}
This property represents the earlier mentioned observation that the error of solutions found by neural networks seems to be highly correlated with the true solution. Intuitively, $\epsilon$-closeness allows approximations to have large errors in the derivatives where the true solution itself has large derivatives. In order to be $\epsilon$-close, approximations must be very accurate at the flatter regions of the true solution.

Given this definition, it is desirable that $\epsilon$ be small. Thus, $\epsilon$ can be used to guide the choice of $\lambda$. To this end, consider a general, linear PDE given by
\begin{equation}
	\begin{cases}
	\begin{aligned}
	\sum_{i=1}^{k_I} \alpha^I_i(\bm{x}) \partial_{\bm{x}}^{\bm{\beta}_i}u(\bm{x}) &= F(\bm{x}), && \text{in } \Omega,\\
	\sum_{i=1}^{k_B} \alpha^B_i(\bm{x}) \partial_{\bm{x}}^{\bm{\gamma}_i}u(\bm{x}) &= G(\bm{x}), && \text{on }\partial\Omega.
	\end{aligned} \label{eq:pde_def_3}
	\end{cases}
\end{equation}
The aim is to derive an upper bound for the losses under the assumption that the approximated solution $\hat{u}$ is $\epsilon$-close to $u$. First, consider the interior loss functional. Applying eq. \ref{eq:interiorLoss} to the PDE results in the interior loss functional
\begin{align}
\begin{split}
\hat{L}_I(u) &= \int_{\Omega} \left|\mathcal{N}(\bm{x}, u)\right|^p d\bm{x}= \int_{\Omega} \left|\sum_{i=1}^{k_N} \alpha^I_i(\bm{x}) \partial_{\bm{x}}^{\bm{\beta}_i}u(\bm{x}) - F(\bm{x})\right|^p d\bm{x}.\label{eq:loss_def_interior_2}
\end{split}
\end{align}	
Note that the solution $u$ satisfies
\begin{equation}
\mathcal{N}(\bm{x}, u) = \sum_{i=1}^{k_N} \alpha^I_i(\bm{x}) \partial_{\bm{x}}^{\bm{\beta}_i}u(\bm{x}) - F(\bm{x}) = 0
\end{equation}
for all $\bm{x}\in\Omega$. Therefore, it follows that
\begin{align}
\begin{split}
	\mathcal{N}(\bm{x}, \hat{u}) 
	= \sum_{i=1}^{k_N} \alpha^I_i(\bm{x}) \left[\partial_{\bm{x}}^{\bm{\beta}_i}\hat{u}(\bm{x}) - \partial_{\bm{x}}^{\bm{\beta}_i}u(\bm{x})\right].
	\end{split}
\end{align}
Substituting the approximation $\hat{u}$ into eq. \ref{eq:loss_def_interior_2} allows the integral to be written as
\begin{align}
	\hat{L}_I(\hat{u}) &= \int_{\Omega} \left|\sum_{i=1}^{k_N} \alpha^I_i(\bm{x}) \left[\partial_{\bm{x}}^{\bm{\beta}_i}\hat{u}(\bm{x}) - \partial_{\bm{x}}^{\bm{\beta}_i}u(\bm{x})\right]\right|^p d\bm{x}.
\end{align}
Because it was assumed that $\hat{u}$ satisfies eq. \ref{eq:epsClose}, the integrand, which will be labeled $l_I(\bm{x}, \hat{u})$, can be bounded by
\begin{align}
	\begin{split}
	l_I(\bm{x}, \hat{u}) &= \left|\sum_{i=1}^{k_N} \alpha^I_i(\bm{x}) \left[\partial_{\bm{x}}^{\bm{\beta}_i}\hat{u}(\bm{x}) - \partial_{\bm{x}}^{\bm{\beta}_i}u(\bm{x})\right]\right|^p\leq \left[\sum_{i=1}^{k_N} \left|\alpha^I_i(\bm{x})\right| \left|\partial_{\bm{x}}^{\bm{\beta}_i}\hat{u}(\bm{x}) - \partial_{\bm{x}}^{\bm{\beta}_i}u(\bm{x})\right|\right]^p\\
	&\leq \left[\sum_{i=1}^{k_N} \left|\alpha^I_i(\bm{x})\right| \epsilon \left|\partial_{\bm{x}}^{\bm{\beta}_i}u(\bm{x})\right| \right]^p= \epsilon^p \left[\sum_{i=1}^{k_N} \left|\alpha^I_i(\bm{x}) \partial_{\bm{x}}^{\bm{\beta}_i}u(\bm{x})\right| \right]^p.
	\end{split}
\end{align}
The resulting upper bound for the interior loss is thus given by
\begin{align}
	\begin{split}
	\hat{L}_I(\hat{u}) &= \int_\Omega l_I(\bm{x}, \hat{u}) d\bm{x} \leq \epsilon^p \int_\Omega \left[\sum_{i=1}^{k_N} \left|\alpha^I_i(\bm{x}) \partial_{\bm{x}}^{\bm{\beta}_i}u(\bm{x})\right| \right]^p d\bm{x}=:\epsilon^p M_I(u).\label{eq:interior_loss_bound}
	\end{split}
\end{align}
One can derive an upper bound for the boundary loss functional in a similar fashion. For the PDE defined in eq. \ref{eq:pde_def_3}, the resulting upper bound for the boundary loss is given by
\begin{align}
	\hat{L}_B(\hat{u}) \leq \epsilon^p \int_{\partial\Omega} \left[\sum_{i=1}^{k_B}\left| \alpha^B_i(\bm{x}) \partial_{\bm{x}}^{\bm{\gamma}_i}u(\bm{x})\right|\right]^p d\bm{x}_\Gamma=:\epsilon^p M_B(u). \label{eq:boundary_loss_bound} 
\end{align}

These inequalities define necessary, but not sufficient, conditions for a solution to be $\epsilon$-close to the true solution. Crucially, these bounds are extremely lenient. For a given $\epsilon$, the upper bound on the loss functional is much larger than a sufficient bound would need to be to guarantee the same absolute error. This discrepancy arises mainly from the assumptions on the distributions of the derivatives. Note that these assumptions result in loss bounds that depend on the true solution of the PDE. In contrast, sufficient conditions to guarantee a certain absolute error would only depend on the PDE itself. This ties in with the same reasoning as used before, i.e. solutions with large derivatives tend to have favorable error distributions, and can therefore have larger errors in their derivatives to achieve the same accuracy. 

From a different perspective, eqs. \ref{eq:interior_loss_bound}, \ref{eq:boundary_loss_bound} can be used to obtain lower bounds on $\epsilon$. Recall that the aim of this section is to find the optimal value of $\lambda$ that would result in the most accurate approximation, i.e. the smallest $\epsilon$ such that the approximation is $\epsilon$-close. Since training algorithms are only aware of the total loss, it makes sense to perform this optimization with respect to some arbitrary, fixed value $\hat{L}(\hat{u})$.

Because $\lambda$ and $\epsilon$ have no direct relation, the lower bounds on $\epsilon$ given in eqs. \ref{eq:interior_loss_bound},  \ref{eq:boundary_loss_bound} should be minimized instead. Rewriting these bounds into a single expression yields
\begin{equation}
    \epsilon \geq \hat{\epsilon} := \left[\max\left\{\frac{\hat{L}_I(\hat{u})}{M_I(u)}, \frac{\hat{L}_B(\hat{u})}{M_B(u)}\right\}\right]^{\frac{1}{p}}.
\end{equation}

While $\hat{\epsilon}$ is not dependent on $\lambda$, it can be bounded from above using the total loss functional. In particular, for the interior loss one finds that
\begin{align}
\begin{split}
\hat{L}_I(\hat{u}) &= \frac{1}{\lambda} \left[\hat{L}(\hat{u}) - (1 - \lambda)\hat{L}_B(\hat{u})\right] \leq \frac{1}{\lambda}\hat{L}(\hat{u}).
\end{split}
\end{align}
Similarly, one finds for the boundary loss that
\begin{align}
\begin{split}
\hat{L}_B(\hat{u}) &= \frac{1}{1 - \lambda} \left[\hat{L}(\hat{u}) - \lambda \hat{L}_I(\hat{u})\right] \leq \frac{1}{1 - \lambda}\hat{L}(\hat{u}).
\end{split}
\end{align}
Thus, $\hat{\epsilon}$ is bounded from above by
\begin{equation}
    \hat{\epsilon} = \left[\max\left\{\frac{\hat{L}_I(\hat{u})}{M_I(u)}, \frac{\hat{L}_B(\hat{u})}{M_B(u)}\right\}\right]^{\frac{1}{p}} \leq \left[\max\left\{\frac{\hat{L}(\hat{u})}{\lambda M_I(u)}, \frac{\hat{L}(\hat{u})}{(1-\lambda)M_B(u)}\right\}\right]^{\frac{1}{p}}.
\end{equation}
Since $\hat{L}(\hat{u})$ is assumed to be constant, the optimal choice for $\lambda$ then becomes
\begin{align}
\begin{split}
	\lambda &= \argmin_\lambda \hat{\epsilon} = \argmin_\lambda\left[\max\left\{\frac{1}{\lambda M_I(u)}, \frac{1}{(1 - \lambda) M_B(u)}\right\}\right].
	\end{split}
\end{align}
The smallest maximum occurs when both terms are equal, i.e. when
\begin{equation}
    \lambda M_I(u) = (1 - \lambda) M_B(u),
\end{equation}
of which the solution is given by
\begin{equation}
	\lambda = \frac{M_B(u)}{M_I(u) + M_B(u)}. \label{eq:optilambda}
\end{equation}
Notice that this choice of $\lambda$ may be very different from the original choice as given in eq. \ref{eq:lambda_orig}. 

Eq. \ref{eq:optilambda} gives a choice of $\lambda$ that is in some sense optimal: for a given loss value $\hat{L}(\hat{u})$, this choice of $\lambda$ results in the smallest $\epsilon$ for which $\hat{u}$ may be $\epsilon$-close to $u$. This choice gives the loss functionals some additional useful properties. Most notably, the relative importance of the interior and boundary losses become scale-invariant. Specifically, when one rewrites a PDE of the form given in eq. \ref{eq:pde_def_1} as 
\begin{align}
	\begin{cases}
	\begin{aligned}
	c_1\mathcal{N}(\bm{x}, u) &= 0 &&\text{in } \Omega,\\
	c_2\mathcal{B}(\bm{x}, u) &= 0 &&\text{on } \partial\Omega,
	\end{aligned}
	\end{cases}
\end{align}
then the ratio
\begin{equation}
	\frac{\lambda \hat{L}_I(\hat{u})}{(1 - \lambda) \hat{L}_B(\hat{u})}
\end{equation}
remains constant for all $c_1, c_2 \neq 0$. This is a very important property. All PDEs come with some inherent scale factors, and even the default choices $c_1 = c_2 = 1$ are difficult to justify. This choice of $\lambda$ renders these scales almost completely irrelevant. 

However, it is important to realize that $\epsilon$-closeness is contingent on some very strong assumptions, some which are easily shown to not hold true in practice. In particular, functions that zero out anywhere in the domain are problematic for this definition, since $\epsilon$-closeness implies zero error for such functions at specific regions. It is likely that the true dynamics of neural networks depend on the behavior of the target function in a small area, rather than at a particular point. 

This does not immediately render $\epsilon$-closeness useless. However, one should carefully examine the PDE before using $\epsilon$-closeness to estimate the difficulty of optimizing the involved objectives. For instance, problems with homogeneous boundary conditions require different methods to properly estimate the difficulties; applying $\epsilon$-closeness to such problems results in zero expected boundary error, leading to an optimal loss weight of $0$. This definition is likely only useful if the behavior of the true solution on the boundary is comparable to its behavior in the interior of the domain. Similarly, adding offsets to linear PDEs can have a major impact on the theoretical optimal loss weight, even though neural networks can compensate for such offsets by simply tuning the bias of the output layer. Therefore, the definition is best suited for problems with solutions with zero mean. 

\section{Magnitude Normalization}	\label{sec:normalization}
The particular choice of $\lambda$ derived in the previous section resulted in a relation between $\lambda$ and the true solution $u$. In many practical applications, the required information about $u$ is unavailable. This section addresses this problem by introducing a heuristic method that approximates the optimal choice as $\hat{u}$ approaches $u$. This heuristic method has one crucial difference compared to the methods discussed so far; $\lambda$ is no longer kept constant. Instead, $\lambda$ is treated as an additional functional to be optimized by the neural network.

The starting point of deriving this heuristic method is eq. \ref{eq:optilambda}. This optimal choice of $\lambda$ depends solely on the true solution of the PDE, leading to a total loss functional given by
\begin{equation}
	\hat{L}(\hat{u}) = \frac{M_B(u) \hat{L}_I(\hat{u})}{M_I(u) + M_B(u)} + \frac{M_I(u) \hat{L}_B(\hat{u})}{M_I(u) + M_B(u)}.
\end{equation}
Since $u$ is unavailable, and $\hat{u}$ is meant to approximate $u$, one could consider approximating this loss functional by
\begin{equation}
	\hat{L}(\hat{u}) = \frac{M_B(\hat{u}) \hat{L}_I(\hat{u})}{M_I(\hat{u}) + M_B(\hat{u})} + \frac{M_I(\hat{u}) \hat{L}_B(\hat{u})}{M_I(\hat{u}) + M_B(\hat{u})}. \label{eq:loss_def_6}
\end{equation}
When $\hat{u}\approx u$, this redefined loss functional behaves similarly to the loss functional of eq. \ref{eq:loss_def_5} with the optimal choice of $\lambda$, since the bounds $M_I$ and $M_B$ will be relatively constant. However, when $\hat{u}$ differs from $u$, the behavior may no longer be similar, as the approximated bounds $M_I(\hat{u})$ and $M_B(\hat{u})$ cannot be treated as constants. 

This behaviour is difficult to analyze without specific information about the PDE one aims to solve. However, some quirks may be identified by considering examples. Many PDEs admit trivial solutions when the boundary conditions are homogeneous. Even though such PDEs are generally only of interest with inhomogeneous boundary conditions, the mere existence of these trivial solutions is problematic. To see why, consider the linear, well-posed PDE given in eq. \ref{eq:trivialproblem}.
\begin{equation} \label{eq:trivialproblem}
\begin{cases}
\begin{aligned}
\mathcal{N}(\bm{x}, u) &= 0, && \text{in } \Omega,\\
u(\bm{x}) &= G(\bm{x}), && \text{on }\partial\Omega,
\end{aligned}
\end{cases}
\end{equation}
with $G(\bm{x}) \neq 0$ and $\mathcal{N}(\bm{x}, 0) = 0$. Clearly, the trivial solution would solve this problem if and only if $G(\bm{x}) = 0$. To show why this is problematic, we consider behavior of the loss functional given in eq. \ref{eq:loss_def_6} for this function. For the interior and boundary losses, it holds that
\begin{align}
	\hat{L}_I(0) = 0,\;\; \hat{L}_B(0) = \int_{\partial\Omega} \left|G(\bm{x})\right|^p d\bm{x}_\gamma > 0.
\end{align}
Similarly, the bounds given in eq. \ref{eq:boundary_loss_bound}-\ref{eq:interior_loss_bound} satisfy
\begin{align}
	M_I(0) = 0,\;\;M_B(0) = \int_{\partial\Omega} \left|G(\bm{x})\right|^p d\bm{x}_\gamma = \hat{L}_B(0).
\end{align}
Therefore, the total loss thus becomes
\begin{align}
\begin{split}
\hat{L}(0) &= \frac{M_B(0) \hat{L}_I(0) + M_I(0) \hat{L}_B(0)}{M_I(0) + M_B(0)}= \frac{\hat{L}_B(0) \cdot 0 + 0 \cdot \hat{L}_B(0)}{\hat{L}_B(0)} = 0.
\end{split}
\end{align}
Thus, a function that violates the boundary conditions is able to bring the loss functional of eq. \ref{eq:loss_def_6} down to zero. Note that this problem can also arise if the boundary magnitude and losses can be zeroed out simultaneously; in that case, the PDE does not need to be satisfied to find a solution with zero loss. Thus, Eq. \ref{eq:loss_def_6} is not a viable loss functional.

It turns out that this issue can be easily addressed while maintaining the relative sizes of the terms $M_B\hat{L}_I$ and $M_I\hat{L}_B$. This can be achieved by multiplying equation \ref{eq:loss_def_6} with a single scale factor, such that the total loss functional is given by
\begin{align}
\begin{split}
\hat{L}(\hat{u}) &= \frac{M_I(\hat{u}) + M_B(\hat{u})}{M_I(\hat{u})M_B(\hat{u})}\left[\frac{M_B(\hat{u}) \hat{L}_I(\hat{u})}{M_I(\hat{u}) + M_B(\hat{u})} + \frac{M_I(\hat{u}) \hat{L}_B(\hat{u})}{M_I(\hat{u}) + M_B(\hat{u})}\right]\\
&= \frac{\hat{L}_I(\hat{u})}{M_I(\hat{u})} + \frac{\hat{L}_B(\hat{u})}{M_B(\hat{u})}. \label{eq:loss_def_7}
\end{split}
\end{align}
The resulting loss functional has a very intuitive interpretation. Each loss functional is normalized by the magnitude of the terms that comprise it. The resulting terms $\frac{\hat{L}_I(\hat{u})}{M_I(\hat{u})}$ and $\frac{\hat{L}_B(\hat{u})}{M_B(\hat{u})}$ can be viewed as relative losses. The resulting method is called \textit{magnitude normalization}.

This loss functional, unlike the functional defined in eq. \ref{eq:loss_def_6}, again possesses the property that the unique global minimizer is the true solution of the PDE. However, it is not clear whether this loss functional possesses additional local minima, which might compromise convergence to the true solution. Local minima thus pose a threat to the stability of the method. 

However, such local minima are characterized by strictly positive loss. In many cases, stability problems can be therefore avoided by pre-training the network; if the neural network is brought to a state with loss lower than one would have at a local minimum, then this local minimum should never be reached. 

To further manage stability problems, some additional modifications are introduced. Many problems come with Dirichlet or Neumann boundary conditions, i.e. known operator values along the boundary. This allows one to compute the value $M_B(u)$ in advance. Using this value instead of approximating it with $M_B(\hat{u})$ can significantly improve stability. Another stability concern arises from source functions. Since $\epsilon$-closeness does not depend on source functions, neither do $M_I$ and $M_B$. However, source functions can significantly inflate the initial value of the loss functional, allowing the network to converge to local minima with higher losses. To prevent this from happening, source functions are included in $M_I$. Close to the true solution, this generally has little effect, as the effective loss weight may only change by up to a factor two. Far away from the true solution, however, the behavior is much more stable. For a PDE of the form given in eq. \ref{eq:pde_def_3} with Dirichlet or Neumann boundary conditions, the resulting magnitude normalized loss functional is given by
\begin{equation}
\small
\begin{split}
\hat{L}(\hat{u}) &= \frac{\int_\Omega \left|\left(\sum_{j=1}^{k_I}\alpha^I_j(\bm{x})\partial^{\bm{\beta}_j}_{\bm{x}} \hat{u}(\bm{x}, \bm{\theta})\right) - F(\bm{x})\right|^p d\bm{x}}{\int_\Omega \left[\sum_{j=1}^{k_I}\left|\alpha^I_j(\bm{x})\partial^{\bm{\beta}_j}_{\bm{x}} \hat{u}(\bm{x}, \bm{\theta})\right| + \left|F(\bm{x})\right|\right]^p d\bm{x}} + \frac{\int_{\partial\Omega} \left[\partial^{\bm{\gamma}}_{\bm{x}} \hat{u}(\bm{x}, \bm{\theta}) - G(\bm{x})\right]^p d\bm{x}_\Gamma}{\int_{\partial\Omega} \left|G(\bm{x})\right|^p d\bm{x}_\Gamma}. \label{eq:loss_def_8}
\end{split}
\end{equation}

\begin{remark}
Magnitude normalization can favor functions with large derivatives, since those functions lead to a large denominator in the loss functional. While this can cause instabilities, it can also aid in solving more difficult problems, as is shown in Section \ref{sec:results}. 
\end{remark}
\section{Method Setup} \label{subsec:methodsetup}

As stated, this section aims to assess the methods by utilizing small neural networks in conjunction with powerful training algorithms. In this section, the network configuration, the evaluation of the loss functionals, and the specifics of the training algorithms are discussed. 

For the sake of generality, only fully connected feedforward neural networks are considered here. The networks have four hidden layers with twenty neurons each unless specified otherwise. This yields $1301 + 20d$ degrees of freedom, where $d$ is the dimension of the PDE. This network size is kept constant throughout this section. 
The hyperbolic tangent is used as the activation function, though alternative choices like sinusoids seem to yield comparable results. Glorot initialization \cite{pmlr-v9-glorot10a} is used to generate the initial weights of the neural networks. Initially, the biases are set to zero.

There are many different algorithms with which these networks can be trained. First-order methods such as Adam are among the most popular. Another prominent training algorithm is the limited-memory version of the Broyden-Fletcher-Goldfarb-Shanno algorithm (L-BFGS) \cite{byrd1995limited}. This is a quasi-Newton method that is able to achieve exceptionally accurate results, at the cost of one major drawback: this algorithm is incompatible with batching, and must be provided with a dataset that is representative of the full solution. The problems that are solved in this work are simple enough for this limitation to not be prohibitive, and therefore this algorithm is the training method of choice in this work. We also tested first-order methods such as Adam \cite{Adam} and stochastic gradient descent, but these methods were always outperformed by L-BFGS. For problems that require datasets that are too large to be processed at once, we recommend using Adam. 

\subsection{Loss Functions}
Thus, we aim to use L-BFGS to optimize the vector $\bm{\theta}$ containing all parameters of the neural network with respect to some objective function. As stated before, the integrals present in the loss functionals in sections \ref{sec:motivation}-\ref{sec:normalization} can be approximated using Monte-Carlo integration. We explicate this for linear PDEs of the form of eq. \ref{eq:pde_def_3}. Recall that the general form of the interior and boundary loss functionals $\hat{L}_I$ and $\hat{L}_B$ is given in eq. \ref{eq:loss_def_interior_2}. Let the output of the neural network at position $\bm{x}$ be given by $\hat{u}(\bm{x}, \bm{\theta})$. Then, after applying a Monte-Carlo approximation, we obtain the interior and boundary loss functions for the parameter vector $\bm{\theta}$ given by
\begin{align}
L_I(\bm{\theta}) &= \frac{1}{n_I}\sum_{i=1}^{n_I} \left|\left(\sum_{j=1}^{k_I}\alpha^I_j(\bm{x}^I_i)\partial^{\bm{\beta}_j}_{\bm{x}} \hat{u}(\bm{x}^I_i, \bm{\theta})\right) - F(\bm{x}^I_i)\right|^p,\label{eq:loss_int_MC}\\
L_B(\bm{\theta}) &= \frac{1}{n_B}\sum_{i=1}^{n_B} \left|\left(\sum_{j=1}^{k_B}\alpha^B_j(\bm{x}^B_i)\partial^{\bm{\gamma}_j}_{\bm{x}} \hat{u}(\bm{x}^B_i, \bm{\theta})\right) - G(\bm{x}^B_i)\right|^p.\label{eq:loss_bound_MC}
\end{align}
Here, the collocation points $\{x^I_i\}_{i=1}^{n_I}$ and $\{x^B_i\}_{i=1}^{n_B}$ are distributed uniformly over $\Omega$ and $\partial\Omega$, respectively. Using eqs. \ref{eq:loss_int_MC}, \ref{eq:loss_bound_MC}, Monte-Carlo approximations of eqs. \ref{eq:loss_def_3} with $p=2$ and eq. \ref{eq:loss_def_5} are given by
\begin{align}
L(\bm{\theta}) &= L_I(\bm{\theta}) + L_B(\bm{\theta}) \label{eq:loss_def_og_MC},\\
L(\bm{\theta}) &= \Omega\lambda L_I(\bm{\theta}) + \partial\Omega(1-\lambda)L_B(\bm{\theta}), \label{eq:loss_def_opt_MC}
\end{align}
respectively. The Monte-Carlo approximation of Eq. \ref{eq:loss_def_8}  is given by
\begin{equation}
\small
\begin{split}
L(\bm{\theta}) &= \frac{\sum_{i=1}^{n_I} \left[\left(\sum_{j=1}^{k_I}\alpha^I_j(\bm{x}_i^I)\partial^{\bm{\beta}_j}_{\bm{x}} \hat{u}(\bm{x}_i^I, \bm{\theta})\right) - F(\bm{x}_i^I)\right]^p}{\sum_{i=1}^{n_I} \left[\sum_{j=1}^{k_I}\left|\alpha^I_j(\bm{x}_i^I)\partial^{\bm{\beta}_j}_{\bm{x}} \hat{u}(\bm{x}_i^I, \bm{\theta})\right| + \left|F(\bm{x}_i^I)\right|\right]^p} + \frac{\sum_{i=1}^{n_B} \left[\partial^{\bm{\gamma}}_{\bm{x}} \hat{u}(\bm{x}^B_i) - G(\bm{x}_i^B)\right]^p}
{\sum_{i=1}^{n_B} \left|G(\bm{x}_i^B)\right|^p}. \label{eq:loss_def_mag_MC}
\end{split}
\end{equation}
These loss functions define three methods of interest that will be considered in our experiments: the original method as introduced in \cite{PhysInformedI}, our proposed modification to this method with an optimal loss weight, and the heuristic approximation of this optimal weight, respectively. To perform the experiments, the norm $p=2$ is used. Note that instead of minimizing some function $f(\bm{\theta})$, one could equivalently minimize the function $g(f(\bm{\theta}))$ for some monotonically increasing $g$. Because the loss functions are supposed to become small during training, the gradients of these loss functions with respect to $\bm{\theta}$ will likely become very small as well. Therefore, instead of minimizing the loss functions directly, the logarithm of the loss functions is minimized. In many cases, this significantly improves the results.

\subsection{Adaptive Collocation Points}
The loss functions depend on the hyperparameters $n_I$ and $n_B$. Determining the required number of collocation points, both on the boundary and in the interior of the domain, is generally challenging, as it depends on the variance of the loss functions, which in turn depends on the state of the neural network. 
We propose an alternative solution: choosing the point counts adaptively, by comparing the training loss to some validation loss. First choose initial values for the point counts $n_I$ and $n_B$ and then generate two different sets of collocation points, each consisting of $n_I$ interior and $n_B$ boundary points. The first set is the training set, and the second set is the validation set. L-BFGS is used to minimize the loss function evaluated on the training set. During training, the loss function is also evaluated on the validation set. If at any point during training the interior or boundary validation loss is more than a factor $q$ larger than either the interior or boundary training loss, then the variance of the respective loss function is too large to be accurately approximated with the number of collocation points used. In this case, the corresponding point count $n_I$ or $n_B$ is doubled, new collocation point sets are generated, and training is resumed. The initial point counts $n_I$ nd $n_B$ are generally set to 512. The precise value of the hyperparameter $q$ seems to have little effect, and $q=5$ seems to be a reasonable trade-off between accuracy and speed. 

\subsection{GPU Acceleration}

Neural networks are well suited for parallelization, as many of the computations that must be performed to compute a training iteration are independent. The study of \cite{NORDSTROM1992260} identifies four different strategies with varying granularity that can be used to parallelize neural network computations. This shows that the computations can be distributed effectively across a large number of cores. 
To exploit this, many libraries have been developed to aid the parallelization of deep learning. Notable examples of such libraries are Theano, Pytorch and Tensorflow. 
In this work, we use Tensorflow version 1.15.0. to perform the computations in parallel. Tensorflow does not explicitly use any of the four parallelization strategies identified in \cite{NORDSTROM1992260}, but rather uses the dataflow paradigm to explicitly construct graphs of operations required for the computations. This exposes the data that is required for each step of the computations, allowing groups of independent operations to be identified. Tensorflow then distributes these operations across the available cores to accelerate the computations. 
Python source code that implements the listed methods to solve a model problem can be found at Github\footnote{\url{https://github.com/remcovandermeer/Optimally-Weighted-PINNs}}. The numerical experiments presented in this work were computed on a Tesla K80 GPU provided by Google's Colaboratory project. 

\section{Experimental Results} 	\label{sec:results}

This section aims to experimentally compare the methods developed in this work to the original method introduced in the study of \cite{PhysInformedI} by applying all methods to solve several different PDEs. Because this work focuses on improving the problem formulation, rather than on fine-tuning hyperparameters of the neural networks and training algorithms, it makes sense to use small neural networks in conjunction with powerful training algorithms. Limiting the network capacity necessitates using this capacity efficiently, which will highlight the differences between the methods. By using powerful training algorithms, we ensure that the obtained solutions accurately reflect the properties of the methods of interest. The details regarding the neural networks and the training algorithms that were used are given in Section \ref{subsec:methodsetup}.

Two new methods were developed in work. Section \ref{sec:modification} introduced loss weights, with the aim to balance the priorities of the neural network, and derived an optimal choice for this loss weight. Section \ref{sec:normalization} introduced a heuristic method for approximating this optimum, which may be of value when limited information about the analytical solution is available. We label the original method \textit{Original}, and our new methods \textit{Optimal Loss Weight} and \textit{Magnitude Normalization}, respectively.

To probe the limits of these three methods, we use them to solve several different model PDEs, which are chosen such that they can easily be made more difficult. Section \ref{subsec:laplace} covers the first two PDEs that we consider: the Laplace and Poisson equations. Because these PDEs are well understood and easy to solve with traditional numerical methods, they are also easy to analyze. Furthermore, these PDEs are commonly used as testbeds. For instance, the study of \cite{2019arXiv190407200D}, as well as the study of \cite{1997physics...5023L} solve boundary value problems of the Poisson equation and report excellent accuracy. However, the problem instances that are solved in these studies are particularly simple. 
Since we aim to probe the limits of the three methods, more challenging variants of these PDEs are also considered here, including higher-dimensional problems and problems with peaks. 

Section \ref{subsec:convecdiffusion} treats the convection-diffusion equation. This PDE becomes extremely challenging to solve with traditional numerical methods if the diffusivity becomes small. As the diffusion rate approaches zero, the solutions start to exhibit boundary layers, which often require prohibitively fine grids and small time steps to solve. Decreasing the diffusion rate thus serves as the main tool to increase the difficulty of these problems.

\subsection{Laplace Equation} \label{subsec:laplace}

This section covers the Laplace equation, which describes the stationary heat equation and forms the basis of many model PDEs. As mentioned, this PDE is well understood and therefore serves as the starting point of the experimental analysis. Here, our focus is on analyzing behavior that is specific to neural networks. Traditional numerical methods are generally only dependent on the PDE and the grid size. However, since neural networks have limited learning capacity based on their size, one would expect that the difficulty of solving a PDE depends, among other things, on the complexity of the true solution. Therefore, choosing boundary conditions that lead to increasingly complex solutions forms the basis of probing the limits of the original and proposed methods. For a $d$-dimensional system, boundary value problems corresponding to the Laplace equation assume the form
\begin{equation}
\begin{cases}
\begin{aligned}
\nabla^2 u(\bm{x}) & = 0 	  &&\text{in } \Omega,\\
u(\bm{x}) & = G(\bm{x})\qquad &&\text{on } \partial\Omega.
\end{aligned}
\end{cases}
\label{eq:laplacedef}
\end{equation}
To keep the arithmetic simple, problems are considered on the $d$-dimensional unit hypercube with boundary conditions that correspond to the eigenfunctions of the Laplace equation, $\hat{u}(\bm{x}) \equiv  \prod_{i=2}^{d} \sin \omega_ix_ie^{-\omega_1 x_1}$,
with $\omega_i$ multiples of $\pi$ and $\omega_1 = \sqrt{\sum_{i=2}^{d}\omega_i^2}$. This gives rise to inhomogeneous boundary conditions at $x_1=0$ and $x_1=1$, and homogeneous boundary conditions on the remaining $2d-2$ boundary hyperplanes. High frequencies are expected to be challenging to learn compared to lower frequencies. Section \ref{subsubsec:laplace2d} covers the problem in two dimensions. Section \ref{subsubsec:lapalcend} covers the Laplace equation in up to six dimensions. 

\subsubsection{Two-Dimensional Problems} \label{subsubsec:laplace2d}

In two dimensions, eigenfunctions are characterized by a single eigenfrequency $\omega$. The three methods are compared for various frequencies ranging from $\omega=\pi$ to $\omega=10\pi$. For these problems, the optimal loss weight follows from the magnitude bounds of eq. \ref{eq:boundary_loss_bound}-\ref{eq:interior_loss_bound}, which are given by
\begin{align}
\begin{split}
	M_I(u) &= \int_\Omega \left[2\omega^2 e^{-\omega x_1}\sin \omega x_2 \right]^2 d\bm{x} = \omega^3(1 - e^{-2\omega}),\\	
	M_B(u) &= \int_{\partial\Omega} \left[e^{-\omega x_1}\sin \omega x_2\right]^2 d\bm{x}_\Gamma = \frac{1}{2}(1 + e^{-2\omega}).
\end{split}
\end{align}
The optimal loss weight is thus approximately given by $\lambda = \frac{M_B(\hat{u})}{M_B(\hat{u}) + M_I(\hat{u})} \approx \frac{1}{1 + 2\omega^3}$.
As a result, optimal loss weights range from approximately 1.58e-2 for $\omega=\pi$ to 1.61e-5 for $\omega=10\pi$. This forecasts significant differences between the original method given in \cite{PhysInformedI} and the proposed ones, especially for higher frequencies. 

To perform the training we use adaptive collocation point counts, starting with 2 interior and boundary collocation points. L-BFGS ran for 20,000 iterations. The relative $L_2$ and $L_\infty$ errors obtained by the three methods are given in table \ref{tab:laplaceres1}.  
\begin{table}[h!]
\footnotesize
	\centering
	\begin{tabularx}{\textwidth}{X||XX|XX|XX}
		\hline
		\multirow{2}{*}{$\omega$} & \multicolumn{2}{c|}{Original} & \multicolumn{2}{c|}{Optimal Loss Weight} & \multicolumn{2}{c}{Magnitude Normalization}\\
		& $L_2$ & $L_\infty$ & $L_2$ & $L_\infty$ & $L_2$ & $L_\infty$ \\\hline\hline
		$1\pi$  	& 9.07e-5 & 1.34e-4 & 2.76e-5 & 1.10e-4 & 2.38e-5 & 6.79e-5\\	
		$2\pi$  	& 1.12e-3 & 2.61e-3 & 6.36e-5 & 1.91e-4 & 2.38e-4 & 5.07e-4\\	
		$4\pi$ 		& 2.11e-2 & 6.44e-2 & 7.27e-4 & 4.63e-4 & 1.75e-3 & 1.31e-3\\	
		$6\pi$  	& 7.50e-1 & 9.77e-1 & 1.01e-3 & 3.39e-4 & 3.77e-3 & 1.50e-3 \\
		$8\pi$  	& 8.25e-1 & 1.21    & 1.84e-2 & 8.71e-3 & 5.95e-3 & 2.08e-3 \\	
		$10\pi$ 	& 2.44    &	1.65    & 1.53e-2 & 9.26e-3 & 7.05e-3 & 2.30e-3 \\\hline
	\end{tabularx}
	\caption{Relative $L_2$ and $L_\infty$ errors of the approximations that were obtained by the three methods for various frequencies. Problems with frequencies higher than $4\pi$ were not solved accurately by the original method. Magnitude normalization and optimal loss weights both resulted in significantly more accurate approximations.}
	\label{tab:laplaceres1}
\end{table}

Table \ref{tab:laplaceres1} shows that there is a significant difference in accuracy between the original method and the methods proposed in this work, which becomes even more pronounced for higher frequencies. In most cases, the proposed methods were at least one order of magnitude more accurate than the original.
Fig. \ref{fig:Laplace_10} illustrates this difference in accuracy for the frequency $\omega=10\pi$. The original method uses a loss weight that is several orders of magnitude larger than the theoretical optimum, and thus focuses excessively on satisfying the PDE. This is reflected by the shape of the obtained approximation: the boundary conditions were not satisfied, but the Laplacian of the solution was extremely small. 

\begin{figure}[h] 
	\centering
	\begin{subfigure}[t]{0.5\textwidth}
		\centering		
		\includegraphics[width=\linewidth]{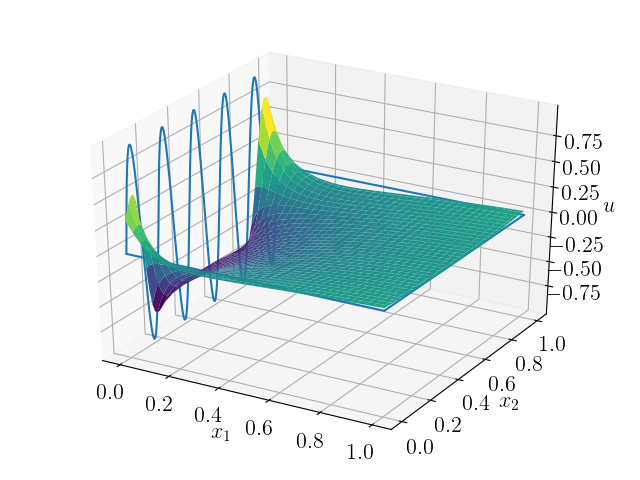}
		\vspace{-0.5cm}
		\subcaption{}
		\label{fig:Laplace_10_Default}
	\end{subfigure}%
	\begin{subfigure}[t]{0.5\textwidth}
		\centering
		\includegraphics[width=\linewidth]{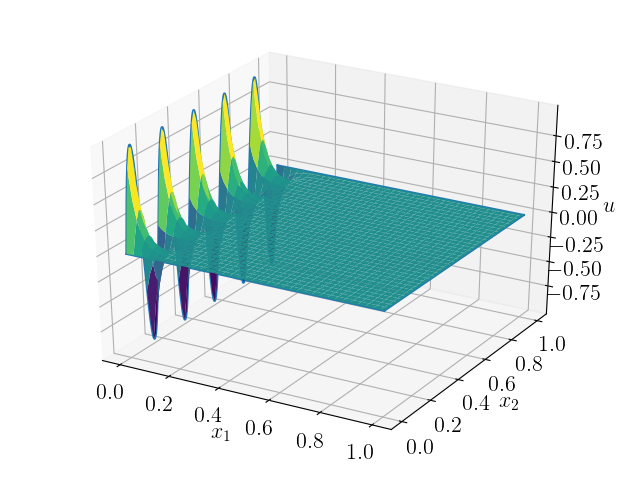}
		\vspace{-0.5cm}
		\subcaption{}
		\label{fig:Laplace_10_Mag}
	\end{subfigure}
	\caption{Boundary conditions and solutions of the two-dimensional Laplace equation with eigenfrequency $\omega=10\pi$ as obtained by the original method (fig. \ref{fig:Laplace_10_Default}) and magnitude normalization (fig. \ref{fig:Laplace_10_Mag}). The original method yielded a solution of which the loss function was completely dominated by the boundary loss, as one would expect based on the optimal value of the loss weight.}
	\label{fig:Laplace_10}
\end{figure}

To better understand the differences between the three methods, the values of the functions $L_I(u)$ and $L_B(u)$ during training are plotted in fig. \ref{fig:Laplace_10_Train}. This highlights that our proposed methods do not yield solutions with lower losses than the original method; instead, the interior loss is sacrificed to better approximate the boundary. The figure also shows that the training behavior of magnitude normalization and optimal loss weights are quite similar. There are two key differences between these two methods. Firstly, the optimal loss weight method uses the optimal loss weight from the very first iteration, while magnitude normalization must first \textit{discover} the overall shape of the true solution before the optimal loss weight is approximated. This is reflected by the loss profile, which shows a slower initial decrease of the boundary loss and a slower initial increase of the interior loss for magnitude normalization compared to optimal loss weights. Secondly, the loss profile of magnitude normalization has significantly larger spikes. The reason for this lies in the additional ways this method has to reduce the total loss: next to directly reducing the interior and boundary losses, the magnitudes can also be increased. Increasing the magnitudes locally to reduce the overall loss generally results in overfitting. Though overfitting was accounted for by using adaptive point counts, the resulting spikes remain visible.

\begin{figure}[h] 
	\centering
	\begin{subfigure}[t]{0.5\textwidth}
		\centering		
		\includegraphics[width=\linewidth]{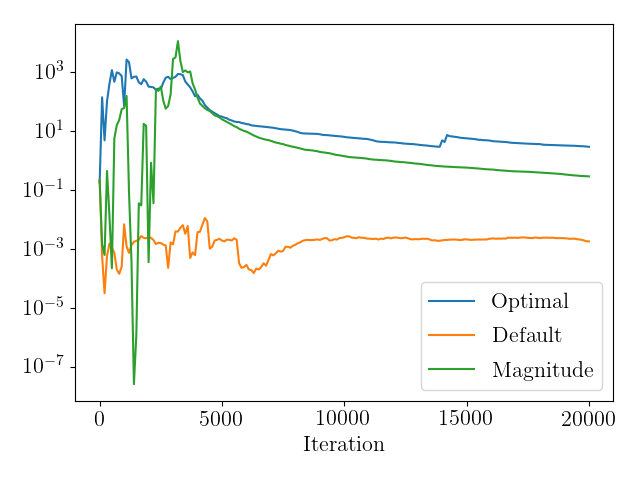}
		\vspace{-0.5cm}
		\subcaption{}
		\label{fig:Laplace_10_Train_Int}
	\end{subfigure}%
	\begin{subfigure}[t]{0.5\textwidth}
		\centering
		\includegraphics[width=\linewidth]{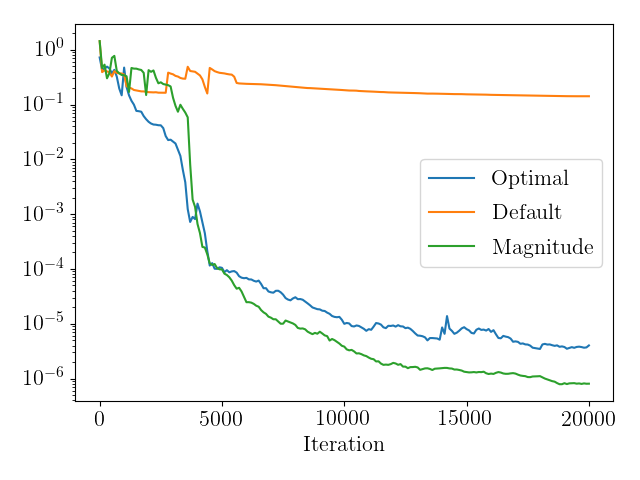}
		\vspace{-0.5cm}
		\subcaption{}
		\label{fig:Laplace_10_Train_Bound}
	\end{subfigure}
	\caption{Loss profiles of the different methods during training for the problem with frequency $\omega=10\pi$. \mbox{Fig \ref{fig:Laplace_10_Train_Int}} shows the interior loss, and fig. \ref{fig:Laplace_10_Train_Bound} shows the boundary loss. Even though the original method resulted in a very inaccurate approximation, the interior loss was much smaller, showing that it is important to assign the right weight to the different loss terms. Note that the neural networks were trained to minimize the logarithm of these losses.}
	\label{fig:Laplace_10_Train}
\end{figure}

To gain some insight in the adaptive point counts that were used to obtain the results, the number of collocation points used for the extreme frequencies are given in table \ref{tab:laplaceresPoints}. This table highlights the strong dependency of the required point counts on the complexity of the true solution, as well as the differences between the three methods. In particular, observe how magnitude normalization tends to use more collocation points. This is caused by its tendency to overfit, as stated earlier.

\begin{table}[h!]
	\centering
	\begin{tabularx}{\textwidth}{X||XX|XX|XX}
		\hline
		\multirow{2}{*}{$\omega$} & \multicolumn{2}{c|}{Original} & \multicolumn{2}{c|}{Optimal Loss Weight} & \multicolumn{2}{c}{Magnitude Normalization}\\\cline{2-7}
		& Interior & Boundary & Interior & Boundary & Interior & Boundary \\\hline\hline
		$1\pi$  	& 512  & 64   & 128  & 128   & 512  & 64   \\		
		$10\pi$  	& 4096 & 128 & 4096 & 512  & 8192 & 512   \\\hline
	\end{tabularx}
	\caption{Adaptive point counts used by the three methods for the extreme problem frequencies. The higher frequency problem required significantly more collocation points, as expected. }
	\label{tab:laplaceresPoints}
\end{table}

We remark that these point counts are increased over the course of the training algorithm, and are therefore tied to the iteration count. Underfitting typically results in a very low point count, while overfitting tends to result in higher point counts. In table \ref{tab:laplaceresPoints} this is reflected in the boundary points that were used by the original method for the frequency $\omega=10\pi$, which severely underfitted the boundary conditions.

Despite small differences in training behavior and collocation point counts, the two proposed methods performed reasonably well for even the most difficult problem tested. However, a decline in accuracy can be observed in table \ref{tab:laplaceres1} for higher frequency problems. This makes sense, as more complex solutions are typically more difficult to learn with neural networks. Typically, either the network capacity must be increased, or the capacity must be utilized better to maintain the same level of accuracy. Since we increased neither the network size nor the number of iterations, a drop in accuracy is expected. Conversely, one would expect accuracy improvements when using larger networks or when training over more iterations.

As observed in table \ref{tab:laplaceres1}, increasing the frequency leads to a drop in accuracy. To confirm that accuracy improvements can be achieved by increasing the network size or the training iterations, the problem with frequency $\omega=10\pi$ is solved with the three methods using a neural network with five layers of 50 neurons. Here, L-BFGS with adaptive points, starting with 2 interior and boundary collocation points, was performed for 50,000 iterations. The resulting errors and final collocation point counts are given in table \ref{tab:laplaceres1_refined}. The collocation point counts generally exceeded the number of points used with the smaller network; this likely resulted from the bigger network being slightly more prone to overfitting. 

\begin{table}[h!]
\footnotesize
	\centering
	\begin{tabular}{l||c|c|c}
		\hline
		& Original & Optimal Loss Weight & Magnitude Normalization\\\hline\hline
		Relative $L_2$ error	   & 9.61e-1 & 1.04e-3 & 1.03e-3\\
		Relative $L_\infty$ error  & 1.13	 & 6.43e-4 & 6.81e-4\\\hline	
		Interior point count       & 16,384  & 16,384  & 32,768\\
		Boundary point count       & 64      & 512     & 512\\\hline
	\end{tabular}
	\caption{The problem with frequency $\omega=10\pi$ solved by the three different methods, using a neural network with five layers of 50 neurons. L-BFGS was run for 50,000 iterations with adaptive point counts, starting with 2 interior and boundary collocation points. The proposed methods achieved excellent accuracy, exceeding the accuracy obtained by the original method when solving the problem with frequency $\omega=2\pi$ using a smaller network.}
	\label{tab:laplaceres1_refined}
\end{table}

To end this section, we briefly consider $\epsilon$-closeness, which forms the foundation on which the proposed methods are built. As mentioned, $\epsilon$-closeness likely only holds approximately. To examine this, the loss and error of the solution obtained by the unscaled method using optimal loss weights are depicted in \mbox{fig. \ref{fig:Laplace_10_ErrorDistribution}}. This figure highlights that while $\epsilon$-closeness is not satisfied exactly, as errors are nonzero where the true solution zeros out, the overall distributions of the errors are shaped similarly to the true solution. Both the loss function, which measures the errors of the second derivatives, and the absolute errors decay significantly in $x_1$, just like the true solution, showing that $\epsilon$-closeness can be a useful concept even in practice.

\begin{figure}[h!] 
	\centering
	\begin{subfigure}[t]{0.5\textwidth}
		\centering		
		\includegraphics[width=\linewidth]{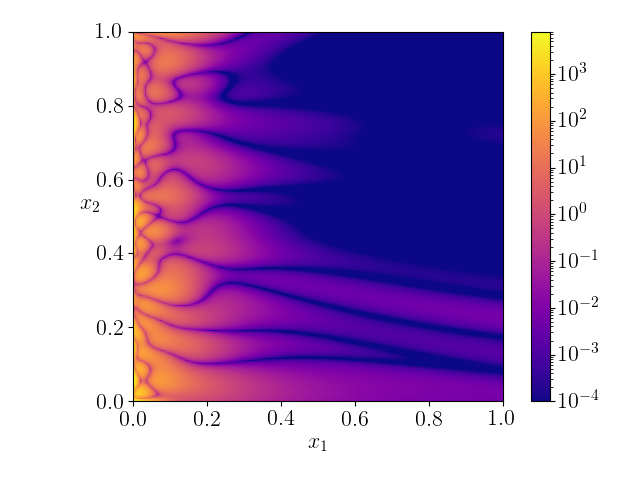}
		\vspace{-0.5cm}
		\subcaption{}
		\label{fig:Laplace_10_Loss}
	\end{subfigure}%
	\begin{subfigure}[t]{0.5\textwidth}
		\centering
		\includegraphics[width=\linewidth]{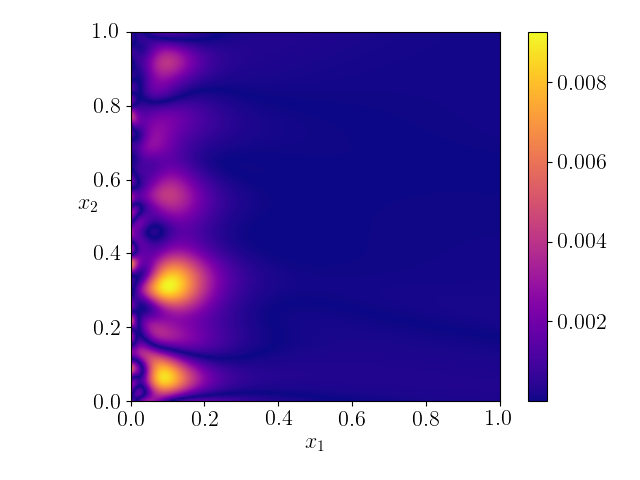}
		\vspace{-0.5cm}
		\subcaption{}
		\label{fig:Laplace_10_Error}
	\end{subfigure}
	\caption{Loss (fig. \ref{fig:Laplace_10_Loss}) and error (fig. \ref{fig:Laplace_10_Error}) distributions of the solution of the problem with frequency $\omega=10\pi$, obtained by using optimal loss weights. Notice how both the loss and the error vanish as the solution becomes flat, highlighting that although $\epsilon$-closeness does not hold exactly, it can be used to predict the overall behavior of the loss function. }
	\label{fig:Laplace_10_ErrorDistribution}
\end{figure}

\subsubsection{Higher-Dimensional Problems} \label{subsubsec:lapalcend}

With basic results established, let us turn to higher-dimensional problems, as they form one of the core motivations for the development of neural network based solvers. As in the previous section, the eigenfunctions are used to generate boundary conditions with easy-to-analyze solutions. Using the Laplace eigenfunctions, $d-1$ frequencies can be chosen freely, while $\omega_1$ is determined by the other frequencies. To prevent the solutions from becoming exponentially complex, $d-2$ frequencies are held constant at $\omega_i = \pi$, and the frequency $\omega_2$ is set to $4\pi$. The focus of this section is on examining the effects of the dimensionality. The results for different dimensionalities are given in table \ref{tab:laplaceres2}. Here, L-BFGS with adaptive point counts, starting with 2 interior and boundary collocation points, was used. 

\begin{table}[h!]
\footnotesize
	\centering
	\begin{tabularx}{\textwidth}{X||XX|XX|XX}
		\hline
		\multirow{2}{*}{$d$} & \multicolumn{2}{c|}{Original} & \multicolumn{2}{c|}{Optimal Loss Weight} & \multicolumn{2}{c}{Magnitude Normalization}\\\cline{2-7}
		& $L_2$ & $L_\infty$ & $L_2$ & $L_\infty$ & $L_2$ & $L_\infty$ \\\hline\hline
		3  	& 2.00e-1 & 4.35e-1 & 2.02e-3 & 4.22e-3 & 2.66e-3 & 4.21e-3\\	
		4  	& 7.29e-1 & 9.68e-1 & 1.00e-2 & 2.49e-2 & 8.66e-3 & 3.54e-2\\	
		5 	& 9.84e-1 & 1.01    & 1.22e-2 & 2.31e-2 & 1.47e-2 & 2.22e-2\\	
		6  	& 1.06    & 9.84e-1 & 4.31e-2 & 4.13e-2 & 4.32e-2 & 4.38e-2\\\hline
	\end{tabularx}
	\caption{Relative $L_2$ and $L_\infty$ errors of the approximations that were found by the three methods for different dimensionalities. The original method has trouble solving even the three-dimensional problem, while the proposed methods gradually become less accurate as the dimensionality grows.}
	\label{tab:laplaceres2}
\end{table}

Adaptive point counts never resulted in more than 16,384 interior or boundary collocation points, even for six-dimensional problems, indicating that high-dimensional problems do not necessarily require massive amounts of data to be solved accurately; instead, this likely only depends on the complexity of the solution. This highlights a big advantage of neural network based methods over traditional numerical methods, which would suffer from an exponential increase in computational cost. 

Like in the previous section, results degraded as the problem difficulty was increased. We again attribute this to the constant network size and number of iterations, which in practice ought to be increased for more difficult problems. To verify this, the six-dimensional problem was solved with a neural network with five layers of 50 neurons over 100,000 iterations, using magnitude normalization with adaptive point counts, starting with 2 interior and boundary collocation points. The obtained solution had a relative $L_2$ error of 5.99e-3 and a relative $L_\infty$ error of 9.88e-3. This again shows that the accuracy can be improved by increasing the available computational resources. 

\subsection{Poisson Equation}

This section covers the Poisson equation, which is the extension of the Laplace equation with a source function. This PDE is analyzed in order to assess the effects of source functions on the proposed methods. The main aim of this section is to show that $\epsilon$-closeness, which by construction does not depend on source functions, is a useful concept even for inhomogeneous PDEs. The Poisson equation has also been studied in other works on neural network based PDE solvers, including the studies of \cite{2019arXiv190407200D, 1997physics...5023L, Xu2018DeepLF}. However, the problems that are solved accurately in these studies are generally characterized by very smooth solutions, whereas our interest lies in solving more complicated problems, as we did in Section \ref{subsubsec:laplace2d}. Only the study of \cite{Xu2018DeepLF} considers a problem with less regular behavior, but the authors did not manage to solve it accurately.

As in the previous section, we aim to solve increasingly complex problems in order to probe the limits of the proposed methods. We restrict ourselves to the 2-dimensional case. In the first part of this section, we consider source functions corresponding to the eigenfunctions of this PDE. These problems can be made more difficult by increasing the frequency. In the second part of this section we solve the problem with a peak source function that was studied but not solved in \cite{Xu2018DeepLF}. 

\subsubsection{Oscillating Solutions}
For a $2$-dimensional system, boundary value\\
 problems corresponding to the Poisson equation can be defined by
\begin{equation}
\begin{cases}
\begin{aligned}
\nabla^2 u(x, y) & = F(x, y)	  &&\text{in } \Omega,\\
u(x, y) & = G(x, y)\qquad &&\text{on } \partial\Omega.
\end{aligned}
\end{cases}
\label{eq:poissondef}
\end{equation}
We consider problems on the unit square $\Omega=[0, 1]\times[0, 1]$. We first consider the eigenfunction problems, with solutions given by $\hat{u}(x, y) = \cos(\omega\pi x)\sin(\omega\pi y)$.
These solutions are prescribed as Dirichlet boundary conditions on the entire boundary of the domain. The source functions that correspond to these eigenfunctions are given by $F(x, y) = -2\omega^2 \cos(\omega\pi x)\sin(\omega\pi y)$. 
This set of problems is parametrized by the frequency $\omega$, which can be utilized to control the difficulty. Similar to the Laplace equation, high frequency problems are expected to be more challenging to solve with these methods than low frequency problems. In addition to this, the optimal loss weights vary based on the frequency. For these problems, the optimal weights follow from 
\begin{align}
M_I(\hat{u}) &= \int_\Omega \left[2\omega^2 \cos(\omega x_1)\sin(\omega x_2)\right]^2 d\bm{x} = \omega^4,\\	
M_B(\hat{u}) &= \int_{\partial\Omega} \left[\cos(\omega x_1)\sin(\omega x_2)\right]^2 d\bm{x}_\Gamma = 1,
\end{align}
such that
\begin{align}
\lambda = \frac{M_B(\hat{u})}{M_B(\hat{u}) + M_I(\hat{u})} = \frac{1}{1 + \omega^4}.
\end{align}
This dependency of $\lambda$ on $\omega$ is even stronger than it was for the Laplace equation, which is why we consider a smaller range of frequencies, given by $\omega\in[\pi, 6\pi]$. For this range of frequencies, the optimal loss weights vary between $\lambda\approx1.02$e-2 for $\omega=\pi$ to $\lambda\approx7.90$e-6 for $\omega=6\pi$. Based on these optimal values, one would again expect significant differences between the three methods for higher frequencies. The problems were solved using the three methods using adaptive collocation point counts, starting with 2 interior and boundary collocation points. 20,000 iterations were performed. Table \ref{tab:poissonres1} contains the relative errors of the obtained approximations. The final numbers of collocation points used for the extreme frequencies are given in \mbox{table \ref{tab:poissonPoints}}. 

\begin{table}[h!]
\footnotesize
	\centering
	\begin{tabularx}{\textwidth}{X||XX|XX|XX}
		\hline
		\multirow{2}{*}{$\omega$} & \multicolumn{2}{c|}{Original} & \multicolumn{2}{c|}{Optimal Loss Weight} & \multicolumn{2}{c}{Magnitude Normalization}\\\cline{2-7}
		& $L_2$ & $L_\infty$ & $L_2$ & $L_\infty$ & $L_2$ & $L_\infty$ \\\hline\hline
		$1\pi$ & 6.66e-4 & 3.51e-3 & 3.47e-5 & 1.86e-4 & 2.82e-5 & 1.16e-4\\	
		$2\pi$ & 3.01e-2 & 1.14e-1 & 1.99e-4 & 5.28e-4 & 2.47e-4 & 5.35e-4\\	
		$4\pi$ & 4.56e-1 & 8.57e-1 & 2.94e-2 & 7.36e-2 & 7.74e-3 & 1.21e-2\\	
		$6\pi$ & 9.82    & 1.27e1  & 2.43e-2 & 1.08e-1 & 3.71e-2 & 5.03e-2 \\\hline
	\end{tabularx}
	\caption{Relative $L_2$ and $L_\infty$ errors of the approximations obtained by the three methods for different $\omega$. The original method failed to solve problems with frequencies of $4\pi$ or higher, while the proposed methods could solve all considered problems. The decline in accuracy for higher frequency problems can likely be explained by the increased complexity of the solutions of these problems.}
	\label{tab:poissonres1}
\end{table}

\begin{figure}[h!] 
\centering
\begin{subfigure}[t]{0.5\textwidth}
	\centering		
	\includegraphics[width=\linewidth]{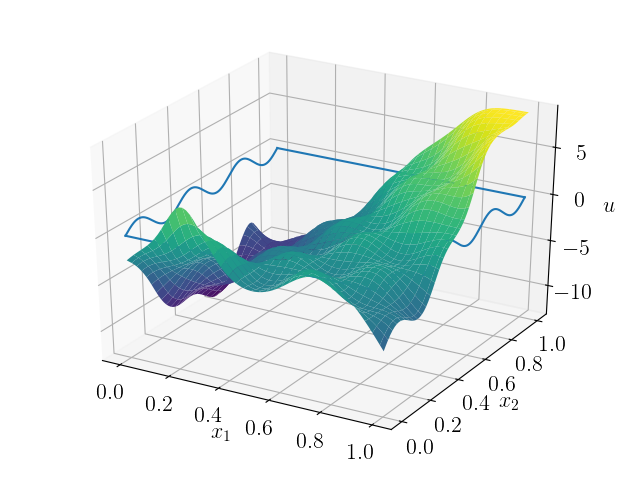}
	\vspace{-0.5cm}
	\subcaption{}
	\label{fig:Poisson_6_Default}
\end{subfigure}%
\begin{subfigure}[t]{0.5\textwidth}
	\centering
	\includegraphics[width=\linewidth]{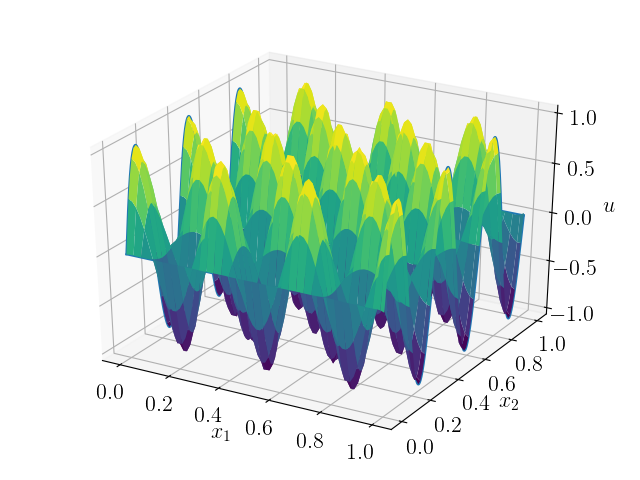}
	\vspace{-0.5cm}
	\subcaption{}
	\label{fig:Poisson_6_Magnitude}
\end{subfigure}
\caption{The solutions to the problem with frequency $\omega=6\pi$ obtained by the original method (fig. \ref{fig:Poisson_6_Default}) and by magnitude normalization (fig. \ref{fig:Poisson_6_Magnitude}). The original method arrived at a solution that satisfies the PDE quite well in the interior of the domain. However, the boundary conditions were practically ignored. In contrast, magnitude normalization yielded a much more accurate solution that satisfied both the boundary conditions and the PDE itself reasonably well.}
\label{fig:Poisson_6}
\end{figure}

These results are similar in nature to our results of the Laplace equation; as expected, the accuracy of the original method drops off much faster than the accuracy of the proposed methods. The approximated solutions obtained by respectively the original and the magnitude normalized methods for the problem with frequency $\omega=6\pi$ are depicted in fig. \ref{fig:Poisson_6}. 
We furthermore observed that the large error of the original method again arises mainly from the boundary conditions not being satisfied, while the two proposed methods were able to fit the boundary conditions, even for the highest frequency considered here. For the Laplace equation, the original method resulted in a solution that interpolated the boundary conditions. However, the Poisson equation is overall more challenging to satisfy than the Laplace equation, as is reflected in the obtained accuracy. As a result, the original method essentially ignored the boundary conditions, since they barely contributed to the total loss. In contrast to the results of the original method, both proposed methods yielded approximations that did not visibly deviate from the true solution, even for the highest frequency considered here. Another difference compared to the results of the Laplace equation is the behavior of adaptive collocation point counts. Although it is not clear why, the presence of a source function seems to increase the rate at which the methods overfit the data, which increases the number of interior collocation points used. This effect was most notable for lower frequencies. The numbers of boundary collocation points used are consistent with the results of the Laplace equation. 

To conclude this section, our results thus show that the concept of $\epsilon$-closeness remains valid even in the presence of source functions. Both optimally weighted method and magnitude normalization make it possible to solve problems with much greater derivatives.
\begin{table}[h!]
\tiny
	\centering
	\begin{tabularx}{0.75\textwidth}{|l||YY|YY|YY|}
		\hline
		\multirow{2}{*}{$\omega$} & \multicolumn{2}{c|}{Original} & \multicolumn{2}{c|}{Optimal Loss Weight} & \multicolumn{2}{c|}{Magnitude Normalization}\\\cline{2-7}
		& Interior & Boundary & Interior & Boundary & Interior & Boundary \\\hline\hline
		$1\pi$  	& 1024 & 64 & 1024 & 64 & 512 & 64 \\		
		$6\pi$  	& 1024 & 8 & 2048 & 256 & 1024 & 256 \\\hline
	\end{tabularx}
	\caption{Adaptive point counts used by the three methods for the extreme problem frequencies. Source functions seem to increase the required number of collocation points. Note that the number of boundary collocation points that the original method used for the high-frequency problem was extremely low. This is indicative of a too low focus on the respective loss term: even as few as eight points did not lead to overfitting. }
	\label{tab:poissonPoints}
\end{table}

Our results thus show that the concept of $\epsilon$-closeness remains valid even in the presence of source functions. Both optimally weighted method and magnitude normalization make it possible to solve problems with much greater derivatives. However, we remark that source functions do affect the stability of magnitude normalization. This did not affect the results presented in this subsection, but does become problematic when one aims to solve problems with even higher frequencies. For such high frequencies, convergence to the true solution is contingent on the initialization of the neural network. In our experiments, these instabilities could always be overcome by using some form of pre-training, i.e. finding a set of network parameters which result in lower loss than the random initialization. It is beyond the scope of this work to address this in detail. 

\subsubsection{Source Functions with Peaks}
Next, we consider the problem mentioned earlier that has been studied in \cite{Xu2018DeepLF}, however without a convincing solution so far. As mentioned, this problem is characterized by the peak in the source function, which causes a peak in the solution. The problem is defined by its solution, which is chosen as 
\begin{equation}
u(x, y) = \sin(\pi x) + e^{-1000\left((x-\frac{1}{2})^2 + (x-\frac{1}{2})^2\right)} -\frac{1}{2}.
\end{equation}
The optimal loss weight of this problem is approximately given by $\lambda\approx 4.45$e-5, which suggests that the proposed methods might be better suited to solving this problem than the original method. To perform the experiments, we again use adaptive collocation point counts, starting with 2 interior and boundary collocation points. For each method we performed 20,000 L-BFGS iterations. The results of the three methods, as well as the final numbers of collocation points, are given in table \ref{tab:poissonres2}.

\begin{table}[h!]
\footnotesize
	\centering
	\begin{tabular}{l||c|c|c}
		\hline
		& Original 	& Optimal Loss Weight    & Magnitude Normalization\\\hline\hline
		relative $L_2$ error		& 1.24e-1 & 4.01e-3 & 2.08e-3\\		
		relative $L_\infty$ error  	& 7.56e-2 & 2.49e-3 & 2.94e-3 \\\hline			
		Interior point count		& 8192 & 8192 			  & 8192\\			
		Boundary point Count		& 8    & 64 			  & 512\\\hline	
	\end{tabular}
	\caption{Relative $L_2$ errors of the approximations obtained for the problem with a peak. The original method is outperformed by the two proposed methods. Note that magnitude normalization used a large number of boundary collocation points. This was caused by an instability that occurred during early training, but was eventually overcome by the method. }
	\label{tab:poissonres2}
\end{table}

Overall, these results line up with the expectations. However, the difference between the three methods is fairly small, despite the very small optimal value of the loss weight. We suspect that the reason that this difference is so small originates from the different characteristics present in the solution; during the early training, the few collocation points present are likely only able to capture the low frequency part of the solution, leading to early approximations that already satisfy the boundary conditions. When more collocation points are added, the peak starts affecting the loss function, which causes the original method to neglect the boundary conditions from then on. However, at this stage, the neural network already fits the boundary conditions rather well. All three methods yielded significantly better results than were obtained in \cite{Xu2018DeepLF}, showing that our approach is promising. 

\begin{figure}[h!] 
	\centering
	\begin{subfigure}[t]{0.5\textwidth}
		\centering		
		\includegraphics[width=\linewidth]{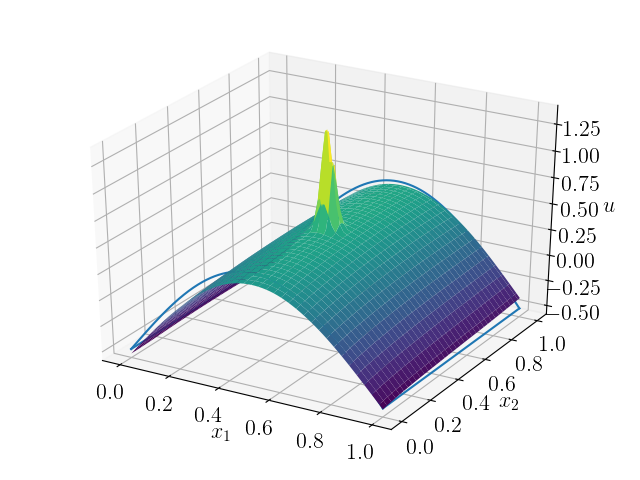}
		\vspace{-0.5cm}
		\subcaption{}
		\label{fig:Poisson_Peak_Default}
	\end{subfigure}%
	\begin{subfigure}[t]{0.5\textwidth}
		\centering
		\includegraphics[width=\linewidth]{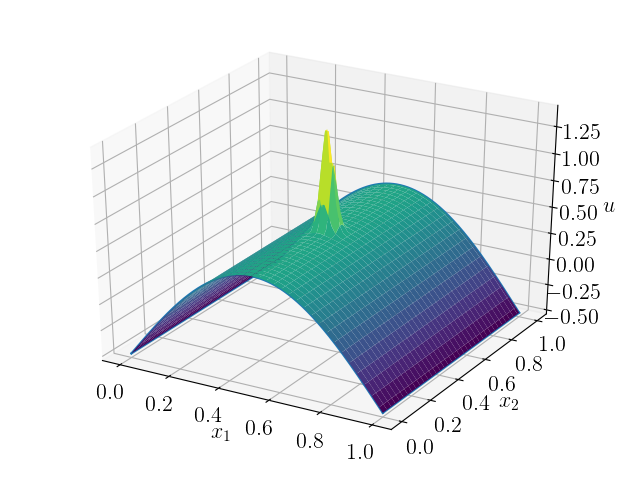}
		\vspace{-0.5cm}
		\subcaption{}
		\label{fig:Poisson_Peak_Magnitude}
	\end{subfigure}
	\caption{The solutions to the problem with a peak obtained by the original method (fig. \ref{fig:Poisson_Peak_Default}) and by magnitude normalization (fig. \ref{fig:Poisson_Peak_Magnitude}). Both approximations have captured the overall behavior of the true solution, and satisfy the PDE in the interior of the domain quite well. However, the original method did not accurately satisfy the boundary conditions.}
	\label{fig:Poisson_Peak}
\end{figure}

To highlight the difference between three methods of interest, the approximations obtained by the original method and by magnitude normalization are depicted in fig. \ref{fig:Poisson_Peak}. Here one can observe that both methods were able to capture the overall behavior of the solution. Furthermore, the figure clearly shows that the error of the original method is caused by a poor satisfaction of the boundary conditions. The behavior of the obtained approximation in the interior of the domain perfectly aligns with the source function. The solution obtained by magnitude normalization did not deviate visibly from the true solution. These results show that these methods are capable of discovering solutions with irregular behavior. To further highlight this, the error of the approximation obtained with magnitude normalization is depicted in fig. \ref{fig:Poisson_Peak_Error}. This figure shows that the error distribution is smooth, even in the region of the peak of the solution. 

\begin{figure}[h!] 
	\centering
	\includegraphics[width=0.5\linewidth]{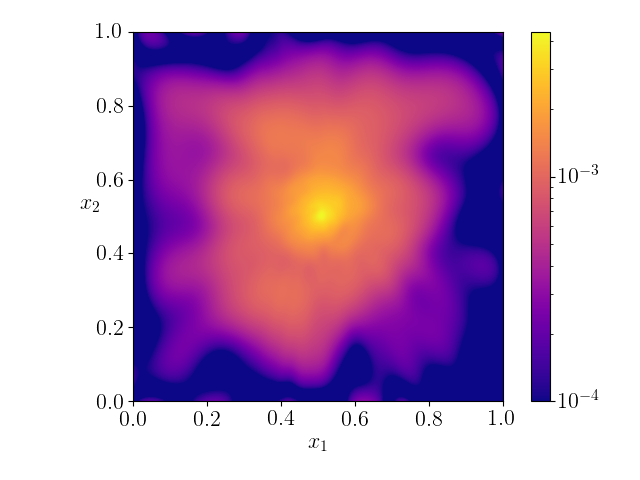}
	\vspace{-0.5cm}
	\caption{The error of the approximation obtained by using magnitude normalization. The distribution of the error is smooth, indicating that no artifacts were introduced by the neural network to handle the sudden change in behavior near the peak. }
	\label{fig:Poisson_Peak_Error}
\end{figure}

\subsection{Convection-Dominated Convection-Diffusion Equation} \label{subsec:convecdiffusion}

So far only nicely elliptic PDEs have been discussed. The experimental results suggest that for hyperbolic PDEs, choosing the loss weight according to the theoretical optimum can have a significant positive effect on the accuracy of these methods. Magnitude normalization, which has qualitatively different behavior, generally results in similar accuracy improvements, but can also be unstable. These instabilities reveal an underlying property of this method: it acts as a driving force towards larger derivatives.

It turns out that this quirk can be exploited to solve a certain class of problems: singularly perturbed problems. These problems are typically challenging to solve with traditional numerical methods because of extreme differences in the behavior of the solutions. One instance of this class of problems is the stationary convection-dominated convection-diffusion equation. This PDE gives rise to solutions with boundary layers, which are small regions where the solution has dramatically different behavior compared to the rest of the solution. For this particular PDE, boundary layers are characterized by extremely large gradients. In this section we briefly show that the tendency of magnitude normalization to find solutions with large gradients makes this method a promising candidate for solving problems with boundary layers. 

In one dimension, the stationary convection-diffusion equation is given by
\begin{equation*}
\begin{cases}
\begin{aligned}
v \frac{d u}{dx} + \alpha \frac{d^2 u}{dx^2} & = 0 &&\text{in } \Omega,\\
u(x) & = G(x) &&\text{on } \partial\Omega.
\end{aligned}
\end{cases}
\end{equation*}
We examine this PDE on the domain $x\in[0, 1]$ with $v=1$, subject to the boundary conditions $u(0) = \frac{1}{2}$, $u(1) = -\frac{1}{2}$. The analytical solution of these problems take the form
\begin{equation}
	u(x) = \frac{e^{-\frac{vx}{\alpha}}}{1 - e^{-\frac{v}{\alpha}}} - \frac{1}{2}.\label{eq:convecdif1d_analytical}
\end{equation}
Eq. \ref{eq:convecdif1d_analytical} implies that the size of the boundary layer is proportional to the diffusivity $\alpha$. Therefore, one would expect that these problems more are challenging to solve as $\alpha\rightarrow 0$. Thus, reducing the diffusivity provides an excellent tool to tune the difficulty of the problem. The three methods are compared for various $\alpha$, ranging from $\alpha=10^{-1}$ down to $\alpha=10^{-4}$. Here, we use adaptive point counts, starting with 2 interior points. The 2 boundary collocation points fully cover the boundary domain, and are hence kept constant. We performed 20,000 iterations for each problem. The results are given in table \ref{tab:statconvecdiffres1}. 

\begin{table}[h!]
\footnotesize
	\centering
	\begin{tabularx}{\textwidth}{X||XX|XX|XX}
		\hline
		\multirow{2}{*}{$\alpha$} & \multicolumn{2}{c|}{Original} & \multicolumn{2}{c|}{Optimal Loss Weight} & \multicolumn{2}{c}{Magnitude Normalization}\\\cline{2-7}
		& $L_2$ & $L_\infty$ & $L_2$ & $L_\infty$ & $L_2$ & $L_\infty$ \\\hline\hline
		$10^{-1}$ & 1.25e-8 & 2.28e-8 & 1.48e-6 & 3.02e-6 & 4.45e-8 & 8.87e-8\\	
		$10^{-2}$ & 1.50e-7 & 4.11e-7 & 2.34e-6 & 4.63e-6 & 2.87e-7 & 5.80e-7\\	
		$10^{-3}$ & 1.02	& 1.14 & 1.14 & 1.98 & 9.25 & 4.56e1\\	
		$10^{-4}$ & 1.01    & 1.20 & 1.15 & 2.00 & 1.91 & 3.51e1 \\\hline
	\end{tabularx}
	\caption{Relative $L_2$ and $L_\infty$ errors of the approximations that were obtained by the three methods for various $\alpha$. For $\alpha\leq 10^{-3}$, no good approximations were found. For larger diffusivity, all three methods yielded very accurate approximations.}
	\label{tab:statconvecdiffres1}
\end{table}

The results given in Table \ref{tab:statconvecdiffres1} show that for $\alpha\geq 10^{-2}$, all three methods yielded very accurate results. Remarkably, the original method outperformed the proposed methods. We suspect that due to the extreme variance of the true solution, a single loss weight is not sufficient to simultaneously deal with the different behaviors inside and outside of the boundary layer. If this is indeed the issue, a local loss weight $\lambda(\bm{x})$ would likely improve the results. However, that is beyond the scope of this work. For smaller $\alpha$, none of the three methods were able to obtain accurate results. 

We will show that it only takes a small modification for magnitude normalization to solve these problems. In particular we observed that magnitude normalization with adaptive point counts resulted in a very large number of collocation points. This was caused by magnitude normalization being unstable for this problem. Eliminating instabilities is not straightforward, and this process typically depends on the nature of the instability. We already mentioned that pre-training is a valid way of overcoming instabilities. However, it may also be helpful to increase the initial number of collocation points, as this prevents the method from introducing spikes into the solution between collocation points. For the particular problem with $\alpha=10^{-4}$, it seems to be sufficient: starting the algorithm with 256 interior collocation points resulted in a relative $L_2$ error of 1.41e-4, and a relative $L_\infty$ error of 8.78e-4. Here, adaptive point counts ultimately resulted in 65,536 interior collocation points, which is a fairly small number considering the size of the boundary layer. This suggests that the poor accuracy of magnitude normalization as shown in table \ref{tab:statconvecdiffres1} was indeed caused by instabilities. The approximated solution when starting with 256 interior collocation points is depicted in fig. \ref{fig:Convec1d_BFGS_Mag_0001}.

\begin{figure}[h!] 
	\centering
	\begin{subfigure}[t]{0.5\textwidth}
		\centering		
		\includegraphics[width=\linewidth]{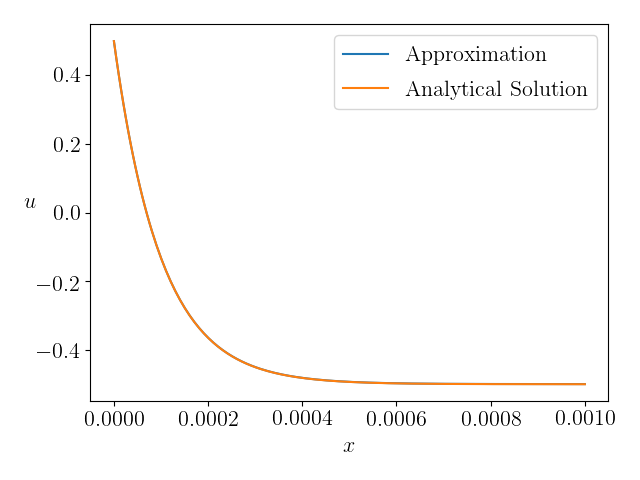}
		\vspace{-0.5cm}
		\subcaption{}
	\end{subfigure}%
	\begin{subfigure}[t]{0.5\textwidth}
		\centering
		\includegraphics[width=\linewidth]{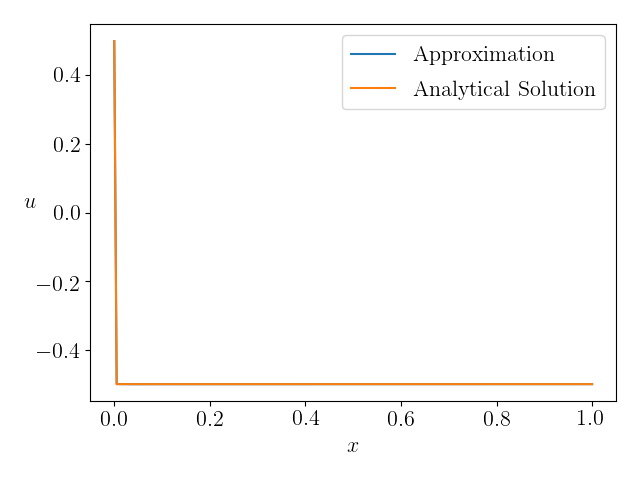}
		\vspace{-0.5cm}
		\subcaption{}
	\end{subfigure}
	\caption{The approximation obtained using magnitude normalization with 256 initial interior collocation points. Even in the boundary layer, no deviations from the true solution are visible.}
	\label{fig:Convec1d_BFGS_Mag_0001}
\end{figure}

We remark that increasing the initial number of collocation points does not enable the other two methods to solve the problems with small $\alpha$. It only slightly improved the obtained accuracy for the problems with large $\alpha$, at significant computational cost. 

\section{Discussion and Conclusion}	\label{sec:conclusions}

The study of \cite{PhysInformedI} introduces an unsupervised learning method that solves PDEs by simultaneously minimizing two objective functions: one encoding the boundary conditions, and one encoding the PDE itself. This method lends itself to solving certain types of problems more efficiently than traditional numerical solvers. For example, neural networks do not suffer from the curse of dimensionality, as they are able to recognize low-dimensional features in high-dimensional space. Furthermore, this method is likely well suited to solving problems for which the underlying geometry is challenging to discretize.

In this work, we generalized this method by introducing a loss weight which compensates for potential imbalances in these two objective functions, which the original method simply adds together with equal weight. To derive an optimal value for this loss weight, the notion of $\epsilon$-closeness was introduced in Section \ref{sec:modification}. This concept was used to predict the deviations of a neural network approximation and its derivatives from the target function, and allowed us to express an optimal value for the loss weight in terms of the true solution of a PDE. We also derived a heuristic method to approximate this loss weight in terms of the approximated solution, which we coined \textit{magnitude normalization}. Furthermore, we introduced a method of adaptively updating the number of collocation points based on the training loss and a test loss.

The significance of $\epsilon$-closeness and the accompanying optimal loss weights was showcased in Section \ref{sec:results}, which contains several model problems that were specifically constructed to have imbalanced objective functions. Using our proposed methods, much better accuracy was obtained for most problems we considered. This shows that neural network based PDE solvers have inherently different behavior compared to traditional numerical methods; their accuracy does not only depend on the PDE one aims to solve, but also on the complexity of the solution. Magnitude normalization was also shown to have useful properties for solving singularly perturbed problems. 

We must, however, mention that magnitude normalization in its current form can be sensitive to the initialization of the neural network, and may be unstable for more difficult problems. These instabilities pose challenges that remain to be solved in future work. The singularly perturbed problems that we considered also indicate that using constant loss weights may not always be ideal, in particular when the solution is characterized by changing behavior. In these cases, a local loss weight $\lambda(\bm{x})$ might prove useful. Such a local loss weight can likely be derived from $\epsilon$-closeness, since this is a local concept. Local loss weights might even be the key to stabilizing magnitude normalization. 

Other open problems that remain have already been identified by the authors of the original study of \cite{PhysInformedI}: error bounds and convergence guarantees are currently non-existent. Given the computational cost of training a neural network, we agree that these methods should not yet be used to solve a single problem instance of a PDE for which traditional methods are available. However, these methods can likely be used to solve parametrized families of problems. While the computational costs involved in training a neural network are large, neural networks are typically cheap to evaluate. Being able to solve many PDEs at once would hence provide a computational reason to prefer neural network based methods over traditional solvers. 

The progress made in this work should bring us one step closer to being able to solve parametrized problems: the theoretical results presented here concern the simultaneous optimization of two general objective functions, and should extend beyond the setting of solving PDEs. A natural next step would be to combine more than two objectives, e.g. by solving multiple PDEs at once. 

\bibliographystyle{siam}
\bibliography{thebib}

\end{document}